\documentclass[11pt,reqno,a4paper]{amsart}

\usepackage{mathrsfs, graphics}
\usepackage{amssymb}
\usepackage[notref,notcite,final]{showkeys}
\usepackage{hyperref}\hypersetup{colorlinks=true, citecolor=blue}
\usepackage{graphicx, epstopdf}\graphicspath{{./immaginiprova/}}
\usepackage{bbm}
\numberwithin{equation}{section}

\newtheoremstyle{mytheorem}
{}
{}
{\it}
{\parindent}
{\bf}
{.}
{ }
{\thmnumber{#2.~}\thmname{#1}\thmnote{~\rm#3}}

\newtheoremstyle{myremark}
{}
{}
{\rm}
{\parindent}
{\bf}
{.}
{ }
{\thmnumber{#2.~}\thmname{#1}\thmnote{~\rm#3}}

\newtheoremstyle{myparagraph}
{}
{}
{\rm}
{\parindent}
{\bf}
{.}
{ }
{\thmnumber{#2.~}\thmname{#1}\thmnote{#3}}

\theoremstyle{mytheorem}

\newtheorem{theorem}[subsection]{Theorem}
\newtheorem{definition}[subsubsection]{Definition}
\newtheorem{lemma}[subsection]{Lemma}
\newtheorem{corollary}[subsection]{Corollary}
\newtheorem{proposition}[subsection]{Proposition}
\newtheorem{question}[subsection]{Question}

\theoremstyle{myremark}
\newtheorem{remark}[subsection]{Remark}

\theoremstyle{myparagraph}

\newtheorem*{parag*}{}

\makeatletter
\def\@secnumfont{\sc}
\def\section{\@startsection{section}{1}%
\z@{1.5\linespacing\@plus .2\linespacing}{.7\linespacing}%
{\normalfont\sc\centering}}
\makeatother

\makeatletter
\def\ps@headings{\ps@empty
 \def\@evenhead{%
  \setTrue{runhead}%
  \normalfont\footnotesize
  \rlap{\thepage}\hfil
  \def\thanks{\protect\thanks@warning}%
  \leftmark{}{}\hfil}%
 \def\@oddhead{%
  \setTrue{runhead}%
  \normalfont\footnotesize\hfil
  \def\thanks{\protect\thanks@warning}%
  \rightmark{}{}\hfil \llap{\thepage}}%
\let\@mkboth\markboth}
\makeatother
\makeatletter
\renewenvironment{proof}[1][\proofname]{\par
  \pushQED{\qed}%
  \normalfont \topsep6\p@\@plus6\p@\relax
  \trivlist
  \itemindent\normalparindent
  \item[\hskip\labelsep
    \bfseries
    #1\@addpunct{.}]\ignorespaces
}{%
  \popQED\endtrivlist\@endpefalse
}
\providecommand{\proofname}{Proof}
\makeatother
\newcommand{\Flat}{\mathbb{F}}
\newcommand{\Mass}{\mathbb{M}}
\newcommand{\TP}{\textbf{TP}}
\newcommand{\OTP}{\textbf{OTP}}

\newcommand{\R}{\mathbb{R}}

\newcommand{\MM}{\mathbb{M}^\alpha}
\newcommand{\N}{\mathbb{N}}

\newcommand{\D}{\mathscr{D}}

\newcommand{\Haus}{\mathscr{H}}
\newcommand{\Leb}{\mathscr{L}}

\newcommand{\M}{\mathscr{M}_+}

\newcommand{\Lip}{\mathrm{Lip}}

\newcommand{\supp}{\mathrm{supp}}

\newcommand{\e}{\varepsilon}
\newcommand{\dV}{d_V\kern-1pt}

\newcommand{\trait}[3]{\vrule width #1ex height #2ex depth #3ex}

\newcommand{\trace}{\mathchoice%
  {\mathbin{\trait{.12}{1.2}{.03}\trait{.8}{0.09}{0.03}}}
  {\mathbin{\trait{.12}{1.2}{.03}\trait{.8}{0.09}{0.03}}}
  {\mathbin{\hskip.15ex\trait{.09}{.84}{0.02}\trait{.56}{.07}{.02}}\hskip.15ex}
  {\mathbin{\trait{.07}{.6}{.01}\trait{.4}{.06}{.01}}}}

\makeatletter
\@addtoreset{footnote}{section}
\makeatother

\begin{document}

	%
\pagestyle{empty}
\pagestyle{myheadings}
\markboth%
{\underline{\centerline{\hfill\footnotesize%
\textsc{Maria Colombo, Antonio De Rosa and Andrea Marchese}%
\vphantom{,}\hfill}}}%
{\underline{\centerline{\hfill\footnotesize%
\textsc{Stability of optimal traffic paths}%
\vphantom{,}\hfill}}}

	%
\thispagestyle{empty}

~\vskip -1.1 cm

	%

\vspace{1.7 cm}

	%
{\large\bf\centering
Improved stability of optimal traffic paths\\
}

\vspace{.6 cm}

	%
\centerline{\sc Maria Colombo, Antonio De Rosa and Andrea Marchese}

\vspace{.8 cm}

{\rightskip 1 cm
\leftskip 1 cm
\parindent 0 pt
\footnotesize

	%
{\sc Abstract.}
Models involving branched structures are employed to describe several supply-demand systems such as the structure of the nerves of a leaf, the system of roots of a tree and the nervous or cardiovascular systems. 
Given a flow (traffic path) that transports a given measure $\mu^-$ onto a target measure $\mu^+$, along a 1-dimensional network, the transportation cost per unit length is supposed in these models to be proportional to a concave power $\alpha \in (0,1)$ of the intensity of the flow.

In this paper we address an open problem in the book \emph{Optimal transportation networks} by Bernot, Caselles and Morel and we improve the stability for optimal traffic paths in the Euclidean space $\mathbb{R}^d$, with respect to variations of the given measures $(\mu^-,\mu^+)$, which was known up to now only for $\alpha>1-\frac1d$. 
 We prove it for exponents $\alpha>1-\frac1{d-1}$ (in particular, for every $\alpha \in (0,1)$ when $d=2$), for a fairly large class of measures $\mu^+$ and $\mu^-$.


\par
}

\tableofcontents

\section{Introduction}
The branched transport problem is a variant of the classical Monge-Kantorovich problem, where the cost of the transportation does not depend only on the initial and the final spatial distribution of the mass that one wants to transfer, but also on the paths along which the mass particles move. It was introduced to model systems which naturally show ramifications, such as roots systems of trees and leaf ribs, the nervous, the bronchial and the cardiovascular systems, but also to describe other supply-demand distribution networks, like irrigation networks, electric power supply, water distribution, etc. In all of the many different formulations of the problem, the main feature is the fact that the cost functional is designed in order to privilege large flows and to prevent diffusion; indeed the transport actually happens on a 1-dimensional network.

To translate this principle in mathematical terms, one can consider costs which are proportional to a power $\alpha \in (0, 1)$ of the flow. Roughly speaking, it is preferable to transport two positive masses $m_1$ and $m_2$ together, rather than separately, because $(m_1+m_2)^\alpha<m_1^\alpha+m_2^\alpha.$
Obviously the smaller is $\alpha$ and the stronger is the grouping effect.

Different costs and descriptions have been introduced in order to model such problem: one of the first proposals came by Gilbert in \cite{Gilbert}, who considered finite directed weighted graphs $G$ with straight edges $e\in E(G)$ ``connecting'' two discrete measures, and a weight function $w : E(G) \to (0, \infty)$.  The cost of $G$ is defined to be:
\begin{equation}\label{e:gilbertenergy}
\sum_{e\in E(G)} w(e)^\alpha \Haus^1(e),
\end{equation}
where we denoted by $\Haus^1$ the $1$-dimensional Hausdorff measure. Later Xia has extended this model to a continuous framework using Radon vector-valued measures, or, equivalantly, 1-dimensional currents, called in this context ``traffic paths'' (see \cite{Xia}). 

In \cite{MSM,BCM1}, new objects called ``traffic plans'' have been introduced and studied. Roughly speaking, a traffic plan is a measure
on the set of Lipschitz paths, where each path represents the trajectory of a single particle. All these formulations were proved to be equivalent (see \cite{BCM} and references therein) and in particular the link between the last two of them is encoded in a deep result, due to Smirnov, on the structure of acyclic, normal 1-dimensional currents (see Theorem~\ref{s-decompcurr}). 

A rich variety of branched transportation problems can be described through these objects: in all of them existence \cite{Xia,MSM,BCM1,BeCaMo,brabutsan,Pegon} and (partially) regularity theory \cite{xia2,MR2250166,DevSol,DevSolElementary,morsant,xiaBoundary,BraSol} are well-established. It is, instead, a challenging problem to perform numerical simulations.

The main reference on the topic is the book \cite{BCM}, which is an almost up-to-date overview on the results in the field. To witness the current research activity on this topic we refer also to the recent works 
 \cite{marmass}, where currents with coefficients in a normed group are used to propose a rephrasing of the discrete problem which could be considered as a convex problem, to \cite{BranK}, which proves the equivalence of several formulations of the urban planning model, including two different regimes of transportation and to \cite{BranRS}, which provides a new convexification of the 2-dimensional problem, used to perform numerical simulations.
 
Other techniques have been recently introduced, with the aim to tackle this and similar problems numerically. For instance \cite{OuSan} provides a Modica-Mortola-type approximation of the branched transportation problem and in \cite{BCF} the authors introduce a family of approximating energies, modeled on the Ambrosio-Tortorelli functional (see also \cite{BLS}). Numerical simulations with a different aim are implemented in the recent works \cite{massoubo} and \cite{BOO}. Here the novel formulations of the Steiner-tree problem and the Gilbert-Steiner problem, introduced in \cite{marmass1} and \cite{marmass}, are exploited to find numerical calibrations: functional-analytic tools which can be used to prove the minimality of a given configuration.

A natural question of special relevance in view of numerical simulations, is whether the optima are stable with respect to variations of the initial and final distribution of mass. In order to introduce this question more precisely and to state our main result, let us give some informal definitions. More technical definitions will be introduced in Section \ref{s:notation} and used along the paper. Nevertheless, the simplified notation introduced here suffices to formulate the question and our main result.\\ 

Given two finite positive measures $\mu^-,\mu^+$ on the set $X:=\overline{B_R(0)}\subset\R^d$ with  $\mu^-(X)=\mu^+(X)$, a traffic path connecting $\mu^-$ to $\mu^+$ is a vector-valued measure $T=\vec T(\Haus^1\trace E)$, supported on a set $E\subset X$, which is contained in a countable union of curves of class $C^1$, having distributional divergence 
$${\mbox{div}}~T=\mu^+-\mu^-.$$
The $\alpha$-mass of $T$ is defined as the quantity
$$\Mass^\alpha(T):=\int_E |\vec T(x)|^{\alpha}d\Haus^1(x).$$
We say that $T$ is an optimal traffic path, and we write $T\in\OTP(\mu^-,\mu^+)$ if 
$$\Mass^{\alpha}(T)\leq\Mass^\alpha(S), \mbox{ for every traffic path } S \mbox { with } \mbox{div}~S=\mu^+-\mu^-.$$

We address the following question about the stability of optimal traffic paths, raised in \cite[Problem 15.1]{BCM}.

\begin{question}\label{question1} Let $\alpha\leq 1-\frac{1}{d}$.
Let $(\mu^-_n)_{n\in\N},(\mu^+_n)_{n\in\N}$ be finite measures on $X$ and for every $n$ let $T_n\in \OTP(\mu^-_n,\mu^+_n)$, with $\MM(T_n)$ uniformly bounded.
Assume that $T_n$ converges to a vector-valued measure $T$ where ${\rm{div}}~T= \mu^+-\mu^-$ and $\mu^\pm$ are finite measures.
Is it true that $T\in \OTP(\mu^-,\mu^+)$?
\end{question}

The threshold
\begin{equation}
\label{ass:threshold}
\alpha=1-\frac{1}{d}
\end{equation}
appears in several contexts in the literature. Firstly, when $\alpha$ is above this value any two probability measures with compact support in $\R^d$ can be connected with finite cost (see Proposition~\ref{irrigability}). Secondly, above this value the answer to the previous question is positive and the minimum cost between two given measures is continuous with respect to the weak$^*$ convergence of measures (see \cite[Lemma 6.11 and Proposition 6.12]{BCM}). Finally, above the threshold interior regularity holds (see \cite[Theorem 8.14]{BCM}) and actually the stability property plays an important role in the proof of such result. The finiteness of the cost, as well as the continuity of the minimum cost, fails for values of $\alpha$ smaller or equal to the value \eqref{ass:threshold} (see \cite{CDRM2} for an example of failure of continuity). Surprisingly enough, the stability of optimal plans still holds, at least under mild additional assumptions. 
The main result of our paper provides a positive answer to the stability question for $\alpha$ below the critical threshold \eqref{ass:threshold}, when the supports of the limit measures $\mu^\pm$ are disjoint and ``not too big''; nothing is instead assumed on the approximating sequence $(\mu_n^\pm)_{n\in\N}$.

\begin{theorem}\label{thm:main}
		Let $\alpha>1-\frac{1}{d-1}$. Let $A^-,A^+\subset X$ be measurable sets and $\mu^-, \mu^+$ be finite measures on $X$ with $\mu^-(X)=\mu^+(X)$, $\supp(\mu^+)\cap \supp(\mu^-)=\emptyset$,
		\begin{equation}\label{ass:supports}
		\Haus^1(A^-\cup A^+)=0 \qquad \mbox{and} \qquad \mu^-(X \setminus A^-)=\mu^+(X \setminus A^+)=0.
		\end{equation}
		Let $  (\mu^-_n)_{n\in \N}, ( \mu^+_n)_{n\in \N}$ be finite measures on $X$ such that $\mu^-_n(X)=\mu^+_n(X)$ and 
		\begin{equation}\label{hp:supp-n-convergence}
		\mu^\pm _n \rightharpoonup \mu^\pm.
		\end{equation}
		For every $n\in \N$ let $T_n\in \OTP(\mu^-_n,\mu^+_n)$ be an optimal traffic path and assume that  there exists a traffic path $T$ and a constant $C>0$ such that 
		$$	T_n\rightharpoonup T \qquad \mbox{and} \qquad \MM(T_n)\leq C.$$
		Then $T$ is optimal, namely
		$$T\in \OTP(\mu^-,\mu^+).$$
	\end{theorem}

\begin{remark}
\begin{enumerate}
\item Notice that in the plane (namely, for $d=2$) our result cover all possible exponents $\alpha\in(0,1)$.
\item The actual notion of traffic path as well as the notion of convergence mentioned in Question \ref{question1} and denoted in Theorem \ref{thm:main} by $T_n\rightharpoonup T$, are slightly different from those used in this introduction (see Subsection \ref{ss:currents}). For our purposes, it is important to observe that the convergence of traffic paths $T_n$ to $T$ implies the convergence of ${\mbox{div }}T_n$ to ${\mbox{div }}T$, weakly in the sense of measures.
\item The assumptions that the supports of $\mu^-$ and $\mu^+$ are disjoint is recurrent in the literature. For example it is assumed in the proof of interior regularity properties of optimal traffic plans (see \cite[Chapter 8]{BCM}). Moreover such hypothesis could be dropped if we assume that either $\mu^-$ or $\mu^+$ are finite atomic measures. However we will not pursue this in the present paper.
\item The restriction that $\mu^\pm$ are supported on $\Haus^1$-null sets is essential for our proof (even though we can relax such assumption in some special case, see \cite{CDRM2}). On the other hand, restrictions on the ``size'' of sets supporting the measures $\mu^\pm$ are recurrent assumptions in previous works (see \cite[Chapter 10]{BCM} and \cite{DevSol}). Requiring \eqref{ass:supports} for supporting Borel sets $A^+$ and $A^-$ rather than for the (closed) supports of $\mu^\pm$, allows one to apply the theorem to more cases; for instance, as soon as the limit measures are supported on any countable set (possibly dense in an open subset of $X$).
\item There is a subtle reason for our choice to use traffic paths, rather than traffic plans, which is related to a known issue about the definition of the cost for traffic plans (see the discussion at the beginning of \cite[Chapter 4]{BCM}). Nevertheless we are able to prove a weaker version of our main result also for traffic plans: roughly speaking one should assume additionally the Hausdorff convergence of the supports of $\mu^\pm_n$ to the supports of $\mu^\pm$. This problem and other versions of the stability results with weaker assumptions on $\mu^\pm$ in some special settings are addressed in \cite{CDRM2}.
\end{enumerate}
\end{remark}

\subsection*{On the structure of the paper}
A few words are worthwhile concerning the organization of the paper. In Section \ref{s:notation} we introduce the main notation and in Section \ref{s:known} we collect some properties of optimal traffic paths which we use extensively through the paper. In particular, in Proposition \ref{p:propr_good_dec} we prove a result about the representation of optimal traffic paths as weighted collections of curves, which paves the way for several new operations on traffic paths introduced in this paper. We conclude Section \ref{s:known} raising the main question on the stability of optimal traffic paths and recalling the results which are already available in the literature. Section \ref{s:sci} requires some explanation: there we prove a result on the lower semi-continuity of the transportation cost. Clearly such property is already used by many other authors. The reason for our attention on that issue is twofold: firstly we want to throw light on a point that is partially overlooked in some previous works (see Remark \ref{remequivalence}), secondly we need a stronger (localized) version of the usual semi-continuity. {Section~\ref{s:ideas} deserves particular attention at a first reading, since it gives a heuristic presentation of the  proof of Theorem \ref{thm:main} and sheds light on several lemmas used therein. We kept the presentation as informal as possible, so that the reader can follow the fundamental ideas of the paper even without being used to the notions and definitions of Section \ref{s:notation}.}
Section \ref{s:prelim} contains several preliminary lemmas, covering results and new techniques which are the ingredients of the proof of the main theorem. Eventually, in Section \ref{s:proof}, we prove Theorem \ref{thm:main}.

\section{Notation and preliminaries}\label{s:notation}
\subsection{Measures  and rectifiable sets}
Given a locally compact separable metric space $Y$, we denote by $\mathscr{M}(Y)$ the set of Radon measures in $Y$, namely the set of (possibly signed) measures on the $\sigma$-algebra of Borel sets of $Y$ that are locally finite and inner regular. We denote also by $\M(Y)$ the subset of positive measures and by $\mathscr P(Y)$ the subset of probability measures, i.e. those poitive measures $\mu$ such that $\mu(Y)=1$. 

We denote by $|\mu|$ the total variation measure associated to $\mu$. The negative and positive part of $\mu$ are the positive measures defined respectively by
$$\frac{|\mu|-\mu}{2} \quad\mbox{and}\quad \frac{|\mu|+\mu}{2}.$$
 For $\mu,\nu\in\M(Y)$, we write $\mu\leq \nu$ in case $\mu(A)\leq \nu(A)$ for every Borel set $A$.
Given a measure \(\mu\) we denote by 
$$\supp (\mu):= \bigcap\{C\subset Y:C \mbox{ is closed and } |\mu|(Y \setminus C)=0\}$$
its \emph{support}. We say that $\mu$ is \emph{supported} on a Borel set $E$ if $|\mu|(Y\setminus E)=0$.
For a  Borel set \(E\),  \(\mu\trace E\) is the  restriction of \(\mu\) to \(E\), i.e. the measure defined by 
$$[\mu\trace E](A)=\mu(E\cap A) \qquad \mbox{for every Borel set $A$.}$$ 
We say that two measures $\mu$ and $\nu$ are mutually singular if there exists a Borel set $E$ such that $\mu=\mu\trace E$ and $\nu=\nu\trace E^c$.

For a measure $\mu\in\mathscr{M}(Y)$ and a Borel map \(\eta :Y\to Z\) between two metric spaces we let \(\eta_\sharp \mu\in\mathscr{M}(Z)\) be the push-forward measure, namely 
$$\eta_\sharp \mu(A):=\mu(\eta^{-1}(A)), \qquad \mbox{for every Borel set }A\subset Z.$$

We use $\Leb^d$ and $\Haus^k$ to denote respectively the $d$-dimensional Lebesgue measure on $\R^d$ and the $k$-dimensional Hausdorff measure, see \cite{SimonLN}.


A  set \(K\subset \R^d\) is said to be \emph{countably \(k\)-rectifiable} (or simply \emph{\(k\)-rectifiable}) if it can be covered, up to an \(\Haus^k\)-negligible set, by countably many $k$-dimensional submanifolds of class \(C^1\). At $\Haus^k$-a.e. point $x$ of a $k$-rectifiable set $E$, a notion of (unoriented) tangent $k$-plane is well-defined: we denote it by ${\rm{Tan}}(E,x)$. 




\subsection{Rectifiable currents}\label{ss:currents} We recall here the basic terminology related to $k$-dimensional rectifiable currents. We refer the reader to the introductory presentation given in the standard textbooks \cite{SimonLN}, \cite{KrantzParks} for further details. The most complete reference remains the treatise \cite{FedererBOOK}. 

A \emph{$k$-dimensional current} $T$ in $\R^d$ is a continuous linear functional on the 
space $\D^k(\R^d)$ of smooth and 
compactly supported differential $k$-forms on $\R^d$. 
Hence the space $\D_k(\R^d)$ of  $k$-dimensional currents in $\R^d$ is endowed with the natural notion of weak$^*$ convergence. For a sequence $(T_n)_{n\in \N}$ of $k$-dimensional currents converging to a current $T$, we use the standard notation $T_n \rightharpoonup T$.
With $\partial T$ we denote the \emph{boundary} of $T$,
that is the $(k-1)$-dimensional current  defined  via 
$$\langle\partial T, \omega \rangle := \langle T, d\omega\rangle \quad\text{ for every } \omega\in \D^{k-1}(\R^d).$$
The \emph{mass} of $T$, denoted by 
$\Mass(T)$, is the supremum of $\langle T, \omega\rangle$ over
all $k$-forms $\omega$ such that $|\omega|\le 1$
everywhere (here with $|\omega|$ we denoted the comass norm of $\omega$). 

By the Radon--Nikod\'{y}m Theorem, a $k$-dimensional current $T$ with finite mass can be identified with the vector-valued measure
$T=\vec{T} \|T\|$ where $\|T\|$ is a finite positive measure
and $\vec{T}$ is a unit $k$-vector field.
Hence, the action of $T$ on a $k$-form
$\omega$ is given by
\[
\langle T, \omega\rangle 
= \int_{\R^d} \langle\omega(x), \vec{T}(x)\rangle d\|T\|(x)
\, .
\]

In particular a $0$-current with finite mass can be identified with a real-valued Radon measure and the mass of the current coincides with the total variation (or mass) of the corresponding measure. We will tacitly use such identification several times through the paper. 

For a current $T$ with finite mass, we will denote by 
$\supp(T)$ its \emph{support}, defined as the support of the associated measure $\|T\|$. A current $T$ is called \emph{normal} if both $T$ 
and $\partial T$ have finite mass; we denote the set of normal $k$-currents in $\R^d$ by $\mathbf{N}_k(\R^d)$.
Given a normal $1$-current $T$, we denote by $\partial_+T$ and $\partial_-T$ respectively the positive and the negative part of the (finite) measure $\partial T$.
It is well-known that, if $T$ is a normal current with compact support and $\partial T=\mu^+-\mu^-$, (where not necessarily $\mu^+$ and $\mu^-$ are mutually singular) it holds
\begin{equation}\label{e:massborduno}
\Mass(\mu^+) = \Mass(\mu^-).
\end{equation}
In particular:
\begin{equation}\label{e:massbord}
\Mass(\partial T)=2\Mass(\partial_- T) = 2\Mass(\partial_+ T).
\end{equation}
Given a Borel set $A \subseteq \R^d$, we define the restriction of a current $T$ with finite mass to $A$ as
$$
\langle T \trace A, \omega\rangle 
:= \int_{A} \langle\omega(x), \vec{T}(x)\rangle d\|T\|(x)
\, .
$$
Notice that the restriction of a normal current to a Borel set is a current with finite mass, but it might fail to be normal.

On the space of $k$-dimensional currents one can define the \emph{flat norm} as
\begin{equation}\label{e:flat}
\Flat(T):=\inf\{\Mass(R)+\Mass(S): T=R+\partial S,\, R \in \D_k(\R^d),\,  S \in \D_{k+1}(\R^d) \}.
\end{equation}
The main reason for our interest on this notion is the fact that the flat norm metrizes the weak$^*$ convergence of normal currents in a compact set with equi-bounded masses and masses of the boundaries. This fact can be easily deduced from \cite[Theorem 4.2.17(1)]{FedererBOOK}.

A $k$-dimensional \emph{rectifiable current} is a current 
$T=T[E,\tau,\theta] $, which can be represented as
\[
\langle T, \omega\rangle 
= \int_{E} \langle\omega(x), \tau(x)\rangle \, \theta(x) d \Haus^k(x)
\, ,
\]
where $E$ is a $k$-rectifiable set, $\tau(x)$ is a unit simple $k$-vector field defined on $E$ which at $\Haus^k$-a.e $x\in E$ spans the approximate tangent space ${\rm{Tan}}(E,x)$ and $\theta: E\to \R$ is a function such that $\int_E |\theta| d \Haus^k<\infty$. We denote by $\mathbf{R}_{k}(\R^d)$ the space of $k$-dimensional rectifiable currents in $\R^d$. Modulo changing sign to the orientation $\tau$, we can always assume that $\theta$ takes non-negative values. We will tacitly make such assumption through the paper, unless we specify elsewhere. It is easy to see that for $T \in \mathbf{R}_{k}(\R^d)$ it holds
\begin{equation}\label{e:mass}
\Mass(T)=\int_{E} \theta(x) d \Haus^k(x);
\end{equation}
in particular, any rectifiable current has finite mass.


\subsection{$\alpha$-mass}
For fixed $\alpha \in [0, 1)$, we define also the \emph{$\alpha$-mass} of a current $T \in \mathbf{R}_k(\R^d)\cup \mathbf{N}_k(\R^d)$ by
\begin{equation}\label{e:alphamass}
\Mass ^\alpha (T ) := \begin{cases} \int_E \theta^\alpha(x) d\Haus^k(x) & \quad \mbox{if $T \in \mathbf{R}_{k}(\R^d)$},\\
+\infty & \quad \mbox{otherwise}.
\end{cases}
\end{equation}
One elementary property of this functional is its \emph{sub-additivity}, namely
\begin{equation}
\label{eqn:mass-subadd}
\MM(T_1+T_2) \leq \MM(T_1)+ \MM(T_2) \qquad \mbox{for every } T_1, T_2 \in \mathbf{R}_k(\R^d)\cup \mathbf{N}_k(\R^d).
\end{equation}
Indeed, the inequality is trivial if $T_1$ or $T_2$ is not rectifiable. In turn, if $T_i= T[E_i, \tau_i, \theta_i]$, $i=1,2$, the multiplicity $\theta$ of $T_1+T_2$ is obtained as the sum of the multiplicities of $T_1$ and $T_2$ with possible signs, so that $\theta \leq \theta_1+\theta_2$. Since moreover the inequality $(\theta_1+\theta_2)^\alpha \leq \theta_1^\alpha+\theta_2^\alpha$ holds for every $\theta_1,\theta_2\in[0,\infty)$, we deduce that
$$
\MM(T_1+T_2) \leq \int_{E_1 \cup E_2} (\theta_1+\theta_2)^\alpha \, d\Haus^k \leq  \int_{E_1 \cup E_2} \theta_1^\alpha+\theta_2^\alpha  \, d\Haus^k = \MM(T_1)+ \MM(T_2).
$$

\subsection{Traffic paths}
Fix $R>0$. From now on, by $X$ we denote the closed ball of radius $R$ in $\R^d$ centered at the origin. Following \cite{Xia} and 
\cite{BCM}, given two positive measures $\mu^-,\mu^+ \in \M(X)$ with the same total variation, we define the set $\TP(\mu^-,\mu^+)$ of the \emph{traffic paths} connecting $\mu^-$ to $\mu^+$ as $$\TP(\mu^-,\mu^+):=\{T\in\mathbf{N}_1(\R^d): \supp(T)\subset X, \partial T=\mu^+-\mu^-\},$$
and the \emph{minimal transport energy} associated to $\mu^-,\mu^+$ as
$$\MM(\mu^-,\mu^+):= \inf \{\MM(T): T \in \TP (\mu^- ,\mu^+)\}.$$

Moreover we define the set of \emph{optimal traffic paths} connecting $\mu^-$ to $\mu^+$ by 
\begin{equation}
	\label{eqn:otp}
	\OTP (\mu^- ,\mu^+):=\{T \in \TP (\mu^- ,\mu^+) : \MM(T)=\MM(\mu^-,\mu^+) \}.
\end{equation}
Given a rectifiable current $T$ with compact support in $\R^d$ and a Lipschitz map $f:\R^d\to\R^m$, we denote by $f_\sharp T$ the push-forward of $T$ according to $f$, i.e the rectifiable current in $\R^m$ defined by
$$\langle f_\sharp T, \omega\rangle:= \langle T, f^\sharp\omega\rangle ,\quad\text{ for every } \omega\in \D^k(\R^m)$$
where $f^\sharp\omega$ is the pull-back of the form $\omega$.

A consequence of the following proposition is that, in order to minimize the $\alpha$-mass among currents with boundary in $X$, it is not restrictive to consider only currents supported in $X$. Indeed the projection onto $X$ reduces the $\alpha$-mass. See also \cite[Lemma~3.2.4 (2)]{depauwhardt}.
\begin{proposition}\label{p:push_forw} 
Let $T\in \mathbf{R}_1(\R^d)$ and let $f:\R^d\to\R^m$ be an $L$-Lipschitz map. Then $\Mass^\alpha(f_\sharp T)\leq L\Mass^\alpha(T)$.
\end{proposition}
\begin{proof}
If $T=T[E,\tau,\theta]$, combining the Area Formula (see \cite[(8.5)]{SimonLN}) and the fact that $(a+b)^{\alpha}\leq a^{\alpha}+b^{\alpha}$ for every $a,b>0$, we get
\begin{equation}
\begin{split}
\Mass^{\alpha}(f_\sharp T)&\leq\int_{f(E)}\left(\int_{f^{-1}(y)}\theta(x)d\Haus^0(x)\right)^\alpha d\Haus^1(y)\\
&\leq\int_{f(E)}\int_{f^{-1}(y)}\theta^\alpha(x) d\Haus^0(x)d\Haus^1(y)\\
&=\int_E J_f(x)\theta^\alpha(x)d\Haus^1(x)\leq L\int_E\theta^\alpha(x) d\Haus^1(x)= L\Mass^{\alpha}(T).
\end{split}
\end{equation}
\end{proof}
\begin{remark}
We notice that, given two measures $\mu^-, \mu^+ \in \M(X)$ with the same total variation and a rectifiable current $R\in \mathbf{R}_{1}(\R^d)$ with $\Mass^\alpha(R)<\infty$ and $\partial R=\mu^+-\mu^-$, there exists $R'\in \mathbf{R}_{1}(X)$ with $\partial R'=\mu^+-\mu^-$ and 
$$\Mass^\alpha(R') \leq \Mass^\alpha(R).$$
More precisely, if $R$ is not supported on $X$, then one can find $R'$ such that $$\Mass^\alpha(R') < \Mass^\alpha(R).$$ The proof of this fact is easily obtained by choosing $R'$ as the push-forward of the current $R$ according to the closest-point projection $\pi$ onto $X$ and applying Proposition \ref{p:push_forw}, observing that $\pi$ has local Lipschitz constant strictly smaller than 1 at all points of $\R^d\setminus X$.
\end{remark}

\begin{remark}[{\bf(Comparison with costs studied in the literature)}]\label{remequivalence} 
The original definition of ``cost''  of a traffic path slightly differs from the $\alpha$-mass defined above. Indeed in \cite[Definition 3.1]{Xia} the author defines the cost of a traffic path as the lower semi-continuous relaxation on the space of normal currents of the functional \eqref{e:gilbertenergy} defined on a class of objects called \emph{polyhedral currents}. In \cite[Section 3]{xia2}, the author notices that, in the class of rectifiable currents, his definition of cost coincides with the $\alpha$-mass defined in \eqref{e:alphamass}. The proof of this fact is only sketched in \cite[Section 6]{White1999} and will be discussed in more detail in \cite{flat-relax}.
To keep the present paper self-contained, in our exposition we prefer not to rely on this fact, but we stick to the notion of cost given by our definition of $\alpha$-mass. We will prove independently in Section \ref{s:sci} that the $\alpha$-mass is lower semi-continuous, together with a localized version of this result that does not appear in the literature.
Since several results in previous works (see for instance Theorem \ref{irrigability}) are first proven for polyhedral chains and then extended by lower semi-continuity, their validity in our setting does not rely on the equivalence between the two costs.


\end{remark}


\section{Known results on optimal traffic paths}\label{s:known}
In this section we collect some of the known properties of optimal traffic paths. The presentation does not aim to be exhaustive, but we only recall the facts used in the proof of our main result.

\subsection{Existence of traffic paths with finite cost}
We begin with the observation that the existence of elements with finite $\alpha$-mass in $\TP(\mu^-,\mu^+)$ is not guaranteed in general. For example in \cite[Theorem 1.2]{DevSol} it is proved that there exists no traffic path with finite $\alpha$-mass connecting a Dirac delta to the Lebesgue measure on a ball if $\alpha\leq 1-\frac 1d$. On the other hand, if the exponent $\alpha$ is larger than such critical threshold, then not only the existence of traffic paths with finite $\alpha$-mass is guaranteed, but one also has a quantitative upper bound on the minimal transport energy.
\begin{theorem}[{\cite[Proposition 3.1]{Xia}}]\label{irrigability}
Let $\alpha > 1-\frac 1d$ and $\mu^-,\mu^+ \in \M(\R^d)$ be two measures with equal mass $M$ supported on a set of diameter $L$. Then
$$\MM(\mu^- ,\mu^+) \leq C_{\alpha,d}M^\alpha L,$$
where $C_{\alpha,d}$ is a constant depending only on $\alpha$ and $d$. 
\end{theorem}

\subsection{Structure of optimal traffic paths}\
An important information about the structure of optimal traffic paths (more in general, about traffic paths of finite $\alpha$-mass) is their rectifiability, which follows immediately from the definition of $\alpha$-mass. Some further piece of information comes from the fact that optimal traffic paths do not ``contain cycles''. A current $T$ with finite mass is called \emph{acyclic} if there exists no non-trivial current $S$ such that
$$\partial S=0 \qquad \mbox{and} \qquad \Mass(T)=\Mass(T-S)+\Mass(S).$$

The following theorem states that optimal traffic paths with finite cost are acyclic. Even though in \cite{PaoliniStepanov} several definitions of cost are considered, the proof of such theorem is given exactly for our cost \eqref{e:alphamass}.

\begin{theorem}[{\cite[Theorem 10.1]{PaoliniStepanov}}] \label{ottimo_acicl}
Let $\mu^-, \mu^+ \in \M(\R^d)$ and $T \in \OTP(\mu^-, \mu^+)$ with finite $\alpha$-mass. Then $T$ is acyclic.
\end{theorem}
The power of this result relies in the possibility to represent acyclic normal $1$-currents as weighted collections of Lipschitz paths. Before stating this result, we introduce some notation.

 We denote by $\Lip$ the space of $1$-Lipschitz curves $\gamma: [0,\infty) \to \R^d$. For $\gamma\in\Lip$ we denote by $T_0(\gamma)$, the value
 $$T_0(\gamma):=\sup\{t:\gamma \mbox{ is constant on }[0,t]\}$$
 and by $T_\infty(\gamma)$ the (possibly infinite) value
 $$T_\infty(\gamma):=\inf\{t:\gamma \mbox{ is constant on }[t,\infty)\}.$$ 
 Given a Lipschitz curve with finite length $\gamma:[0,\infty)\to\R^d$,
 we call $\gamma(\infty):=\lim_{t\to\infty}\gamma(t)$.
 We say that a curve $\gamma\in\Lip$ of finite length is \emph{simple} if $\gamma(s)\neq\gamma(t)$ for every $T_0(\gamma)\leq s<t\leq T_\infty(\gamma)$ such that $\gamma$ is non-constant in the interval $[s,t]$.
 
 To a Lipschitz simple curve with finite length $\gamma:[0,\infty)\to\R^d$,
 we associate canonically the rectifiable $1$-dimensional current 
 $$R_\gamma:=[{\rm{Im}}(\gamma),\frac{\gamma'}{|\gamma'|},1].$$
 It follows immediately from \eqref{e:mass} that
 \begin{equation}
 \Mass(R_\gamma)=\Haus^1(\rm{Im}(\gamma))
 \end{equation}
 and it is easy to verify that
\begin{equation}
\label{eqn:r-gamma-boundary}
\partial R_\gamma=\delta_{\gamma(\infty)}-\delta_{\gamma(0)}.
\end{equation}
Since $\gamma$ is simple, if it is also non-constant, then $\gamma(\infty) \neq \gamma(0)$ and $\Mass(\partial R_\gamma)=2$.

In the following definition, we consider a class of normal currents that can be written as a weighted superposition of Lipschitz simple curves with finite length.
\begin{definition}[\bf{(Good decomposition)}]
Let $T\in \mathbf{N}_1(\mathbb{R}^d)$ and let $\pi \in \M(\Lip)$ be a finite nonnegative measure, supported on the set of curves with finite length, such that 
\begin{equation}
\label{eqn:buona-dec}
T=\int_{\Lip} R_\gamma d \pi (\gamma),
\end{equation}
in the sense of \cite[Section 2.3]{AlbMar}.

We say that $\pi$ is a good decomposition of $T$ if $\pi$ is supported on non-constant, simple curves and satisfies the equalities
\begin{equation}
\label{eqn:buona-dec-mass-T}
\Mass(T) 
= \int_{\Lip} \Mass(R_\gamma) d \pi(\gamma)
= \int_{\Lip} \Haus^1({\rm{Im}}(\gamma)) d \pi(\gamma)
  \, ; 
\end{equation}
\begin{equation}
\label{eqn:buona-dec-mass-boundaryT}
\Mass(\partial T) 
= \int_{\Lip} \Mass(\partial R_\gamma) d \pi(\gamma)
= 2 \pi({\Lip})
  \, .
\end{equation}
\end{definition}

Concretely, \eqref{eqn:buona-dec} means that, representing $T$ as a vector-valued measure $\vec T \| T\|$,
 for every smooth compactly supported vector field $\varphi: \R^d \to \R^d$ it holds
\begin{equation}
\label{eqn:good-dec-operativa}
\int_{\R^d} \varphi \cdot\vec T \,d\| T\|=
\int_{\Lip} \int_{0}^\infty \varphi(\gamma(t)) \cdot \gamma'(t)\, dt \, d \pi(\gamma)
\end{equation}
The following theorem, due to Smirnov (\cite{Smirnov93}), shows that any acyclic, normal,  1-dimensional current has a good decomposition.
\begin{theorem}[{\cite[Theorem 5.1]{PaoliniStepanov1}}]
\label{s-decompcurr}
Let $T=\vec{T} \|T\| \in \mathbf{N}_1(\R^d)$ be an acyclic normal $1$-current.
Then there is a Borel finite measure $\pi$ on $\Lip$ such that $T$ can be decomposed as
$$T=\int_{\Lip} R_\gamma d \pi (\gamma)$$
and $\pi$ is a good decomposition of $T$.

%
\end{theorem}



In the following proposition we collect some useful properties of good decompositions. Further properties will be given in Proposition~\ref{p:propr_good_dec_2}.

\begin{proposition}[\bf{(Properties of good decompositions)}]\label{p:propr_good_dec}
If $T \in \mathbf{N}_1(\mathbb{R}^d)$ has a good decomposition $\pi$ as in \eqref{eqn:buona-dec}, the following statements hold:

\begin{enumerate}
\item The positive and the negative parts of the signed measure $\partial T$ are
\begin{equation}
\label{buona-dec-boundary}\partial_- T = \int_{\Lip}\delta_{\gamma(0)} d \pi (\gamma)
\qquad \mbox{and} \qquad
\partial_+ T = \int_{\Lip}\delta_{\gamma(\infty)} d \pi (\gamma).
\end{equation}
\item If $T= T[E, \tau, \theta]$ is rectifiable, then
\begin{equation}
\label{eqn:dens-acycl}
\theta(x) = \pi(\{\gamma: x \in {\rm{Im}}(\gamma) \}) \qquad \mbox{for $\Haus^1$-a.e. $x\in E$.}
\end{equation}
	\item For every $\pi' \leq \pi$ the representation
	 \begin{equation}
	\label{eqn:Tprimo}
	T' := \int_{\Lip} R_\gamma  d\pi'( \gamma )
	 \end{equation}
	is a good decomposition of $T'$; moreover, if $T= T[E, \tau, \theta]$ is rectifiable, then $T'$
	can be written as $T'=T[E, \theta',\tau]$ with $ \theta' \leq \min\{\theta, \pi'(\Lip)\}$.

	\item If $\Mass^{\alpha}(T)<\infty$, for every $\e>0$ there exists $\delta: = \delta(T, \e)>0$ such that for every $\pi' \leq \pi$ with $\pi'(\Lip) \leq \delta$ we have
	\begin{equation}
	\label{eqn:t'decomp}
	\MM (T') \leq \e,
	\end{equation}
	where $T'$ is defined by \eqref{eqn:Tprimo}.
\end{enumerate}

\end{proposition}
\begin{proof}
{\it Proof of (1).} It follows from the expression in \eqref{eqn:buona-dec}, from the linearity of the boundary operator and from \eqref{eqn:r-gamma-boundary} that
$$
\partial T =  \int_{\Lip} \partial R_{\gamma} d \pi(\gamma) = \int_{\Lip} \delta_{\gamma(\infty)}  \, d \pi(\gamma)- \int_{\Lip}  \delta_{\gamma(0)} \, d \pi(\gamma)=: S_\infty - S_0.
$$
By the subadditivity of the mass and by \eqref{eqn:buona-dec-mass-boundaryT}
\begin{equation*}
	\begin{split}
\Mass( S_\infty) + \Mass(S_0)
&\leq\int_{\Lip}\Mass(\delta_{\gamma(\infty)}) d \pi(\gamma) + \int_{\Lip}\Mass(\delta_{\gamma(0)}) d \pi(\gamma)
\\&
=\int_{\Lip} \Mass( \partial R_\gamma) d \pi(\gamma)
= \Mass(\partial T) = \Mass(S_\infty-S_0)
	\end{split}
\end{equation*}
From this, we deduce that equality holds in the previous chain of inequalities and that there is no cancellation between $S_\infty$ and $S_0$, namely, they are mutually singular measures. This, in turn, implies that they represent the positive and negative part of the measure $\partial T= S_\infty-S_0$.

%

{\it Proof of (2).} We compute, for every smooth compactly supported test function $\phi:\R^d \to \R$,
$$\int_{\R^d} \phi\theta d\Haus^1\trace E=\int_{\Lip}\left(\int_{\R^d} \phi \mathbbm{1}_{\text{Im}\gamma}d\Haus^1\trace E\right) d\pi=\int_{\R^d}\phi\left(\int_{\Lip} \mathbbm{1}_{\text{Im}\gamma}d\pi \right)d\Haus^1\trace E,$$
where in the first equality we used \cite[Theorem 5.5 (iii)]{AlbMar}, which states that \eqref{eqn:buona-dec} induces an analogous equality between the associated positive measures
, and the fact that $\pi$-a.e. $\gamma$ is simple.\\

{\it Proof of (3).} We write $T= T'+(T-T')$ and, since $T-T'$ is ``parametrized'' by $\pi-\pi'$, we have that
\begin{equation}\label{puppa}
\Mass(T')\leq \int_{\Lip} \Mass(R_\gamma) d \pi'(\gamma), \quad \mbox{and} \quad \Mass(T-T')\leq  \int_{\Lip} \Mass(R_\gamma) d (\pi-\pi')(\gamma).
\end{equation}
 We conclude that
\begin{equation}
\begin{split}
\Mass(T) &\leq \Mass(T')+ \Mass(T-T')
\\
& \leq \int_{\Lip} \Mass(R_\gamma) d \pi'(\gamma)+ \int_{\Lip} \Mass(R_\gamma) d (\pi-\pi')(\gamma) = \int_{\Lip} \Mass(R_\gamma) d \pi(\gamma).
\end{split}
\end{equation}
Since $\pi$ represents a good decomposition of $T$, by \eqref{eqn:buona-dec-mass-T} it follows that equality must hold at each step in the previous inequality. In particular, from \eqref{puppa}, we deduce that
$$\Mass(T') 
= \int_{\Lip} \Mass(R_\gamma) d \pi'(\gamma).$$
The same argument applied to the current $\partial T'$ leads to the proof that the property \eqref{eqn:buona-dec-mass-boundaryT} holds for the good decomposition of $T'$.\\

Since the decomposition \eqref{eqn:Tprimo} is good, then, by the formula \eqref{eqn:dens-acycl}, we get that for $\Haus^1$-a.e. $x\in E$
\begin{equation*}
\begin{split}
\theta'(x)&=  \pi'(\{\gamma: x \in {\rm{Im}}(\gamma) \})
\\&\leq \min\big\{ \pi(\{\gamma: x \in {\rm{Im}}(\gamma) \}), \pi'(\Lip) \big\}=\min\big\{\theta(x), \pi'(\Lip)\}.
\end{split}
\end{equation*}
	This concludes the proof of (3).

{\it Proof of (4).} By the previous point, applied to the good decomposition of $T'$ given in \eqref{eqn:Tprimo}, it follows that 
$$\theta'(x) \leq \min\{ \theta, \delta \}.$$
Therefore 
$$
\MM(T') \leq \int_E  \min\{ \theta(x), \delta \}^ \alpha \, d\Haus^1(x)
$$
and the right-hand side converges to $0$ as $\delta \to 0$ by the Lebesgue dominated convergence Theorem.
\end{proof}

%
%

\subsection{Stability of optimal traffic paths}
The present paper addresses Question~\ref{question1}, which we can now rephrase in rigorous terms as follows.

For every $n\in \N$, let $\mu^-_n,\mu^+_n \in \M(X)$ with the same mass and let $T_n\in \OTP(\mu^-_n,\mu^+_n)$, with $\MM(T_n)$ uniformly bounded.
Assume 
$$
T_n\rightharpoonup T, \quad \mbox{ and }\quad \mu_n^\pm\rightharpoonup \mu^\pm
$$
where $\partial T= \mu^+-\mu^-$ and $\mu^\pm\in\M(X)$.
Is it true that $T\in \OTP(\mu^-,\mu^+)$?

The answer is relatively simple for $\alpha  \in (1-1/d,1]$, relying on the fact that the minimal transport energy $\MM(\nu_n,\nu)$ metrizes the weak$^*$-convergence of probability measures $\nu_n\rightharpoonup \nu$, as stated in the following lemma. 

\begin{lemma}[{\cite[Lemma 6.11]{BCM}}]\label{usefullemma}
Let $\alpha>1-\frac{1}{d}$ and $(\nu_n)_{n \in \N}\subset \mathscr P(X)$ be a sequence of probability measures weakly converging to $\nu \in \mathscr P(X)$. Then we have that
$$\lim_{n\to \infty}\MM(\nu_n,\nu)= 0.$$
\end{lemma}

From Lemma \ref{usefullemma} one can easily deduce the following stability result for optimal traffic paths.

\begin{theorem}[{\cite[Proposition 6.12]{BCM}}]\label{stability}
Let $\alpha>1-\frac{1}{d}$. Assume that $(\mu^-_n)_{n \in \N},(\mu^+_n)_{n \in \N} \subset \mathscr{P}(X)$ converge (weakly in the sense of measures) respectively to $\mu^-,\mu^+ \in \mathscr{P}(X)$. Let $T_n\in \OTP(\mu^-_n,\mu^+_n)$ satisfying
$$\sup_{n\in \N}\MM(T_n)< \infty.$$
If $T_n\rightharpoonup T$ for some current $T$, then $T\in \OTP(\mu^-,\mu^+)$.
\end{theorem}

Indeed, assuming by contradiction that Theorem~\ref{stability} does not hold for a sequence $T_n \rightharpoonup T$, we find a contradiction by considering an energy competitor for $T_n$ ($n$ large enough) as follows. We take the optimal transport $T_{opt}$ for the limit problem and we add  two traffic paths of arbitrarily small energy that connect respectively $\mu_n^-$ to $\mu^-$, and $\mu^+$ to $\mu_n^+$. This strategy fails for $\alpha \leq 1- \frac 1 d$, since Lemma~\ref{usefullemma} does not hold below the critical threshold (an example of such phenomenon is provided in \cite{CDRM2}). For this reason, we develop in the following sections a more involved strategy to prove the stability of optimal traffic paths.

\section{Lower semi-continuity of the $\alpha$-mass}\label{s:sci}
This section is devoted to the proof of a lower semi-continuity result for the $\alpha$-mass. The statement will be split in two parts. On one side, we prove the lower semi-continuity for normal currents, which for example allows one to prove the classical existence of optimal traffic paths in \eqref{eqn:otp} (see \cite[Proposition 3.41]{BCM}). On the other side, our strategy of proof of Theorem~\ref{thm:main} requires to work with rectifiable currents with boundary of possibly infinite mass, obtained as restriction of normal rectifiable currents to Borel sets. Therefore for rectifiable currents we prove a localized version of the usual lower semi-continuity.
\begin{theorem}\label{lsc}
Let $k\geq 0$, $\alpha\in (0,1]$,
$(T_n)_{n\in\N}
$ be a sequence of $k$-dimensional currents in $X$, and $T$ be a $k$-dimensional current with 
$$\lim_{n\to \infty} \Flat
(T_n - T ) = 0.$$

\begin{enumerate}
\item
If the $T_n$'s and $T$ are rectifiable and $A$ is an open subset of $X$, then 
\begin{equation}\label{e:lsc_res}
\Mass^\alpha(T\trace A ) \leq  \liminf_{n \to \infty} \MM (T_n\trace A).
\end{equation}
\item 
If $T_n$ and $T$ are normal and
$$\sup_{n\in \N}\big\{\Mass(T_n)+\Mass(\partial T_n)\big\}<+\infty,$$
then
\begin{equation}\label{e:lsc}
\Mass^\alpha(T) \leq  \liminf_{n \to \infty} \MM (T_n).
\end{equation}
\end{enumerate}

\end{theorem}

Using Theorem \ref{lsc}(2) and the compactness of normal currents (see \cite[4.2.17(1)]{FedererBOOK}), the existence of optimal transport paths  in \eqref{eqn:otp} follows via the direct method of the Calculus of Variations.
\begin{corollary}\label{existence}
	Let $\alpha \in (0,1]$. Given two measures $\mu^-,\mu^+ \in \M (X)$ such that $\MM(\mu^- ,\mu^+) <+\infty$, there exists a current $T\in \OTP(\mu^-,\mu^+)$. 
\end{corollary}

The proof of the first part of Theorem~\ref{lsc} employs a characterization of rectifiability by slicing.
The proof of the second point is carried out by slicing our rectifiable currents and reducing the theorem to the lower semi-continuity of $0$-dimensional currents, following some ideas in \cite[Lemma 3.2.14]{depauwhardt}. For this reason, we need to recall some further preliminaries on the slicing of currents. Let $k \leq d$, let $I(d,k)$ be the set of multi-indices of order $k$ in $\R^d$, i.e. the set of $k$-tuples $(i_1,\ldots,i_k)$ with
$$1\leq i_1<\ldots<i_k\leq d,$$
let $\{e_1,\ldots,e_d\}$ be the standard orthonormal basis of $\R^d$, and let $V_{I}$ be the $k$-plane spanned by $\{e_{i_1},\ldots,e_{i_k}\}$ for every $I=(i_1,\ldots,i_k)\in I(d,k)$. Given a $k$-plane $V$, we denote $p_V$ the orthogonal projection on $V$. If $V=V_I$ for some $I$, we simply write $p_I$ instead of $p_{V_I}$. Given a current $T\in \mathbf{N}_k(\R^d)$ with compact support, a Lipschitz function $p:\R^d\to\R^k$ and $y\in\R^k$, we denote by $\langle T,p,y\rangle$ the 0-dimensional \emph{slice} of $T$ in $p^{-1}(y)$ (see \cite[Section 4.3]{FedererBOOK} or \cite[Section 28]{SimonLN} for the case $k=1$). In this paper, we will employ the notion of slicing only to apply two deep known results (contained in Theorem~\ref{thm:rectif-by-slicing} and Lemma~\ref{selodicesimoncicredo}). The following theorem shows that the rectifiability of a current is equivalent to the rectifiability of a suitable family of slices.


\begin{theorem}[{\cite{white_rect}}]\label{thm:rectif-by-slicing}
Let $T \in \mathbf{N}_k(\R^d)$. 
Then $T \in \mathbf{R}_k(\R^d)$ if and only if
$$\langle T,p_I,y\rangle \mbox{ is rectifiable for every $I\in I(d,k)$ and for $\Haus^k$-a.e. $y \in V_I$}.$$ 
\end{theorem}

%

By $Gr(d,k)$ we denote the Grassmannian of $k$-dimensional planes in $\R^d$ and by $\gamma_{d,k}$ we denote the Haar measure on $Gr(d,k)$, i.e. the unique probability measure on $Gr(d,k)$ which is invariant under the action of orthogonal transformations (see \cite[Section 2.1.4]{KrantzParks}).

In the following lemma, we collect some known properties of slices and their behaviour with respect to the $\alpha$-mass and the flat norm. The bounds \eqref{eqn:a-mass-slices} and \eqref{eqn:flat-slices} below are proved in \cite[Corollary 3.2.5(5) and Remark 3.2.11]{depauwhardt} respectively.
 The integral-geometric equality is a consequence of \cite[3.2.26;
2.10.15; 4.3.8]{FedererBOOK} (see also \cite[(21)]{depauwhardt}).

\begin{lemma}\label{selodicesimoncicredo}
Let $R \in \mathbf{R}_k(\R^d)$ and $N \in \mathbf{N}_k(\R^d)$. Then 
for every $ V \in Gr(d,k)$ we have
\begin{equation}
\label{eqn:a-mass-slices}
\int_{\R^k} \Mass^\alpha(\langle R, p_V, y\rangle) \, dy \leq \Mass^\alpha(R),
\end{equation}
\begin{equation}
\label{eqn:flat-slices}
\int_{\R^k} \Flat(\langle N, p_V, y\rangle) \, dy \leq \Flat (N).
\end{equation}
Moreover, there exists $c=c(d,k)$ such that the following integral-geometric equality holds: 

\begin{equation}\label{e:int_geom}
\Mass^{\alpha}(R)=c\int_{Gr(d,k)\times\R^k}\Mass^{\alpha} \big(\langle R,p_V,y\rangle \big) d(\gamma_{d,k}\otimes\Haus^k)(V,y).
\end{equation}

\end{lemma}

\begin{proof}[Proof of Theorem~\ref{lsc}(1)]
{	{\it Step 1: the case $k=0$.} Since a $0$-dimensional rectifiable current $T=T[E,1,\theta]$ is a signed, atomic measure, we write
	$$T \trace A = \sum_{i\in \N} \theta_i \delta_{x_i}$$
	for $(x_i)_{i\in \N}  \subseteq \R^d$ distinct and for $(\theta_i)_{i\in \N} \subseteq \R$ (with possible signs).
	Fix $\e>0$ and let $I \subseteq \N$ be a finite set such that
\begin{equation}
	\label{eqn:finite}
		\MM(T \trace A)-  \sum_{i\in I} |\theta_i |^\alpha \leq \e \qquad \mbox{if }\MM(T \trace A)<\infty
\end{equation}
and 
\begin{equation}
\label{eqn:finite-infin}
 \sum_{i\in I} |\theta_i |^\alpha \geq \frac 1\e \qquad \mbox{otherwise}.
\end{equation}
	Up to reordering the sequences $(x_i)_{i\in \N}$ and $(\theta_i)_{i\in \N}$, we may assume that $I = \{1,...,N\}$ for some $N:= N(\e)$. Set
	$$r := \frac{1}{4} \min\Big\{ \min\{d(x_i,x_j): 1 \leq i <j \leq N\} , \min\{d(x_i,A^c): 1 \leq i \leq N\} \Big\}.$$
	Since $\lim_{n\to \infty} \Flat(T_n - T ) = 0$, then $T_n \rightharpoonup T$ weakly in the sense of measures. Hence 
	for every $i \in \{1,...,N\}$ 
	\begin{equation}\label{dp4}
	\Mass(T\trace  B(x_i,r)) \leq \liminf_{n \to \infty}  \Mass(T_n\trace  B(x_i,r)), \qquad \mbox{for every } i \in \{1,...,N\}.
	\end{equation}	
	By \eqref{dp4} and the elementary inequality
	$ \big(\sum_{i\in \N} |a_i| \big)^\alpha \leq  \sum_{i\in \N} |a_i|^\alpha$ for any $(a_i)_{i\in \N}\subseteq \R,$
	we deduce that for every $i \in \{1,...,N\}$
	\begin{equation}
	\begin{split}
|\theta_i|^\alpha &\leq \big(\Mass(T \trace B(x_i,r)\big)^\alpha \leq \liminf_{n \to \infty} \big( \Mass(T_n\trace  B(x_i,r))\big)^\alpha  
\\&\leq \liminf_{n \to \infty} \MM(T_n\trace  B(x_i,r)).
	\end{split}
	\end{equation}
	Adding over $i$ and observing that the balls $B(x_i,r)$ are disjoint by the choice of $r$, we find that 
	$$ \sum_{i\in I} |\theta_i |^\alpha \leq \liminf_{n \to \infty} \sum_{i=1}^N\MM(T_n\trace  B(x_i,r)) \leq \liminf_{n \to \infty} \sum_{i=1}^N\MM(T_n\trace A).
	$$
	By \eqref{eqn:finite} (or \eqref{eqn:finite-infin} in the case that $\MM(T \trace A)=\infty$) and since $\e$ is arbitrary, we find \eqref{e:lsc_res}.
}

	{\it Step 2 (Reduction to $k=0$ through integral-geometric equality).} We prove now  Theorem~\ref{lsc}(1) for $k > 0$.
	Up to subsequences, we can assume
	$$\lim_{n \to \infty} \MM (T_n\trace A)= \liminf_{n \to \infty} \MM (T_n\trace A) .$$ 
	Integrating in $V\in Gr(d,k)$ the second inequality in Lemma \ref{selodicesimoncicredo} we get
	$$\lim_{n \to \infty} \int_{Gr(d,k)\times \R^k}\Flat(\langle T_n-T,p_V,y\rangle) d(\gamma_{d,k} \otimes \Haus^k)(V,y) \leq \lim_{n \to \infty} \Flat(T_n-T) = 0.$$
	Since the integrand $\Flat(\langle T_n-T,p_V,y\rangle)$ is converging to $0$ in $L^1$
	, up to subsequences, we get
	$$\lim_{n \to \infty}\Flat(\langle T_n-T,p_V,y\rangle)  = 0 \qquad 
	\mbox{for $\gamma_{d,k} \otimes \Haus ^k$-a.e. $(V, y) \in Gr(d,k)\times\R^k$}.$$
	We conclude from Step 1 that
	$$\Mass^\alpha(\langle T,p,y\rangle\trace A ) \leq  \liminf_{n \to \infty} \MM (\langle T_n,p,y\rangle\trace A) .$$
	By \cite[(5.15)]{Ambrosio2000}, for $\Haus^k$-a.e. $y$
	\begin{equation}
	\label{eqn:restr-commuta}
	\langle T\trace A,p_V,y\rangle = \langle T,p_V,y\rangle\trace A.
	\end{equation}
	%
	%
	%
By \eqref{eqn:restr-commuta}, we get the inequality
	\begin{equation}
	\label{eqn:semic-restr}
	\Mass^\alpha(\langle T\trace A ,p_V,y\rangle) \leq  \liminf_{n \to \infty} \MM (\langle T_n\trace A ,p_V,y\rangle) .
	\end{equation}
	The conclusion follows applying twice the integral-geometric equality \eqref{e:int_geom}. Indeed, using the semi-continuity proved for $k=0$ and Fatou's lemma, we get
	\begin{equation}
	\begin{split}
	\Mass^{\alpha}(T\trace A )&=c\int_{Gr(d,k)\times\R^k}\Mass^{\alpha}\big(\langle T\trace A ,p_V,y\rangle\big) d(\gamma_{d,k}\otimes\Haus^k)(V,y)\\
	&\overset{\eqref{eqn:semic-restr}}{\leq} c\int_{Gr(d,k)\times\R^k}\liminf_{n\to \infty}\Mass^{\alpha} \big(\langle T_n\trace A ,p_V,y\rangle \big) d(\gamma_{d,k}\otimes\Haus^k)(V,y)\\
	&\leq c\liminf_{n\to \infty} \int_{Gr(d,k)\times\R^k}\Mass^{\alpha} \big(\langle T_n\trace A ,p_V,y\rangle \big) d(\gamma_{d,k}\otimes\Haus^k)(V,y)\\
	&=\liminf_{n\to \infty}\Mass^{\alpha}(T_n\trace A ).
	\end{split}
	\end{equation}
This concludes the proof of Step 2, so the proof of Theorem \ref{lsc}(1) is complete.
\end{proof}

In order to prove Theorem \ref{lsc}(2), the only property which is missing at this stage is the fact that a normal, non-rectifiable $k$-current cannot be approximated with rectifiable currents with uniformly bounded mass, $\alpha$-mass, and mass of the boundary. This is proved in the following lemma.
\begin{lemma}\label{lemma:limit-rect}
Let $(T_n)\subset \mathbf{R}_k(\R^d)$ and let us assume that
$$\sup_{n\in \N}\{\Mass(T_n)+\Mass(\partial T_n)+\Mass^\alpha( T_n)\}\leq C<+\infty.
$$
If $\lim_{n\to \infty} \Flat
(T_n - T ) = 0$ for some $T \in \mathbf{N}_k(\R^d)$, then $T$ is in fact rectifiable. 
\end{lemma}
\begin{proof}
{\it Step 1: the case $k=0$}. We prove the lemma for $k=0$, recalling  that a $0$-dimensional rectifiable current $T=T[E,\tau,\theta]$, with $\tau(x)=\pm 1$, is an atomic signed measure (i.e. a measure supported on a countable set). 
 More precisely, we prove the following claim: let $(T_n)_{n\in\N}$ be a sequence of $0$-rectifiable currents $T_n =T[E_n,\tau_n,\theta_n]$ such that $\lim_{n\to \infty} \Flat
(T_n - T ) = 0$ for some $T \in \mathbf{N}_0(\R^d)$ and $\Mass^\alpha(T_n) \leq C$ for some $C>0$; then $T$ is $0$-rectifiable.

Indeed, 
fix $\delta>0$. For any $n\in \N$
\begin{equation*}
\begin{split}
\Mass (T_n \trace &\{ x: \theta_n(x)<\delta \} ) 
= \int_{E_n \cap \{\theta_n<\delta \} } \theta_n(x) \, d \Haus^k(x)
\\&
\leq \delta^{1-\alpha}\int_{E_n \cap \{\theta_n<\delta \} } \theta_n(x)^\alpha \, d \Haus^k(x) \leq {\Mass^\alpha(T_n)}\delta^{1-\alpha} \leq C \delta^{1-\alpha}.
\end{split}
\end{equation*}
Therefore, up to subsequences the measure $T_n \trace \{ x: \theta_n(x)\geq\delta \} $ converges to a discrete measure $T_1$ (indeed the support of the measures $T_n \trace \{ x: \theta_n(x)\geq\delta \}$ consists of a finite number of points, which is uniformly bounded with respect to $n$, due to the bound on $\MM(T_n)$), and the sequence $(T_n \trace\{ x: \theta_n(x)<\delta \})_{n\in\N}$ converges to a signed measure $T_2$ of mass less or equal than $C \delta^{1-\alpha}$.

By the arbitrariness of $\delta$, we conclude that the measure $T_2$ has arbitrarily small mass and that the measure $T_1$ is purely atomic. Since $T= T_1+ T_2$, the statement follows.

{\it Step 2}. We prove the claim for $k>0$. 

We apply the inequalities in Lemma~\ref{selodicesimoncicredo} to our sequence $(T_n)_{n\in\N}$ to deduce that 
\begin{equation}\label{eqn:mass-alpha-slices}
\int_{\R^k} \Mass^\alpha(\langle T_n, p_I, y\rangle) \, dy \leq \Mass^\alpha(T_n) \leq C. 
\end{equation}
$$
\lim_{n\to \infty} \int_{\R^k} \Flat(\langle T_n-T, p_I, y\rangle) \, dy \leq C \lim_{n\to \infty} \Flat (T_n-T) = 0.
$$
Since the sequence of non-negative functions $( \Flat(\langle T_n-T, p_I, \cdot \rangle))_{n \in \N}$ converges in $L^1(\R^k)$ to $0$, up to a (not relabelled) subsequence, we get the pointwise convergence
$$\lim_{n \to \infty} \Flat(\langle T_n-T, p_I, y\rangle) = 0 \qquad \mbox{for }\Haus^{k}\mbox{-a.e. } y\in \R^k.$$
Moreover, by Fatou lemma and \eqref{eqn:mass-alpha-slices} we know that for every $I \in I(d,k)$
$$
\int_{\R^k} \liminf_{n\to \infty} \Mass^\alpha(\langle T_n, p_I, y\rangle) \, dy 
\leq \liminf_{n\to \infty} \int_{\R^k} \Mass^\alpha(\langle T_n, p_I, y\rangle) \, dy <\infty. 
$$
Therefore, we have that
$$\liminf_{n\to \infty} \Mass^\alpha(\langle T_n, p_I, y\rangle)  < \infty \qquad \mbox{for }\Haus^{k}\mbox{-a.e. } y\in \R^k.$$
Hence we are in the position to apply Step 1 to a.e.\ slice $\langle T_n, p_I, y\rangle$ to a $y$-dependent subsequence and deduce that
$$\langle T, p_I, y\rangle \mbox{ is $0$-rectifiable for }\Haus^{k}\mbox{-a.e. } y\in \R^k, \; I \in I(d,k).$$
Finally, we employ Theorem~\ref{thm:rectif-by-slicing} to infer that this property of the slices implies that $T$ is rectifiable.
\end{proof}

\begin{proof}[Proof of Theorem \ref{lsc}(2)]
Let $(T_n)\subset\mathbf{N}_k(\R^d)$ and $T \in \mathbf{N}_k(\R^d)$ be such that $\lim_{n\to\infty}\Flat(T_n-T)=0$. If $T$ is rectifiable, then \eqref{e:lsc} follows by Theorem \ref{lsc}(1) and the fact that non-rectifiable currents have infinite $\alpha$-mass. Otherwise if $T$ is non-rectifiable, then \eqref{e:lsc} follows from Lemma \ref{lemma:limit-rect}.
\end{proof}

\section{Ideas for the proof of Theorem~\ref{thm:main}}\label{s:ideas}
Since the proof of Theorem~\ref{thm:main} develops some new geometric ideas in order to construct a suitable competitor for a minimization problem, we introduce informally the strategy in this section, assuming some significant simplifications, before entering the technical details of the actual argument. At the end of this section of heuristics we give some hints on how to remove the further assumptions.
\bigskip

We can easily reduce to the case that $\mu^\pm, \mu^\pm_n \in \mathscr{P}(X)$. By contradiction, we assume that there exists a sequence $T_n \rightharpoonup T$ of optimizers such that $T$ is not an optimizer, namely there exists $T_{opt}$ and $\Delta>0$ with
$$\MM(T_{opt}) \leq \MM(T) - \Delta, \qquad \partial T_{opt} = \partial T = \mu^+-\mu^-.$$
We aim to find a contradiction by defining a suitable competitor $\tilde T_n$ for $T_n$ for some $n$ large enough, that ``almost follows'' $T_{opt}$ instead of $T$, and satisfies the estimates
$$\MM(\tilde T_{n}) \leq \MM(T_n) - \frac{\Delta}{8}, \qquad \partial \tilde T_{n} = \partial T_n = \mu^+_n-\mu^-_n.$$

{\it (1) Covering of $A^\pm$.} 
First, we choose a countable covering of the sets $A^\pm$ supporting $\mu^\pm$, denoted by $\{ B^\pm_i = B^\pm(x_i,r_i)\}_{i\in \N}$ (see Figure~(\ref{fig:covering-tn}a)) such that
\begin{equation}
\label{eqn:chosen-small}
\sum_{i=1}^\infty r_i, \quad \MM\Big( T \trace  \bigcup_{i=1}^{\infty} B^\pm_i\Big), \quad \mbox{and}\quad \MM\Big( T_{opt} \trace  \bigcup_{i=1}^{\infty} B^\pm_i\Big) \mbox{ are arbitrarily small}.
\end{equation}
This choice is made possible by the assumption that the measures $\mu^\pm$ are supported on sets of $\Haus^1$-measure $0$ and by the fact that $\Mass^\alpha$ is absolutely continuous with respect to $\Haus^1$.
We also select a finite number $N^\pm$ such that 
\begin{equation}
\label{eur:fuori-poco}
\mu^\pm \Big( \Big(\bigcup_{i=1}^{N^\pm} B^\pm_i\Big)^c\Big) \mbox{ is small}.
\end{equation}
For simplicity, in this section we make the assumption that the balls $B_i^\pm$ are pairwise disjoint and that the coverings are finite, namely the quantity in \eqref{eur:fuori-poco} is $0$. 
\begin{figure}[h]
\begin{center}
\includegraphics[scale=0.170]{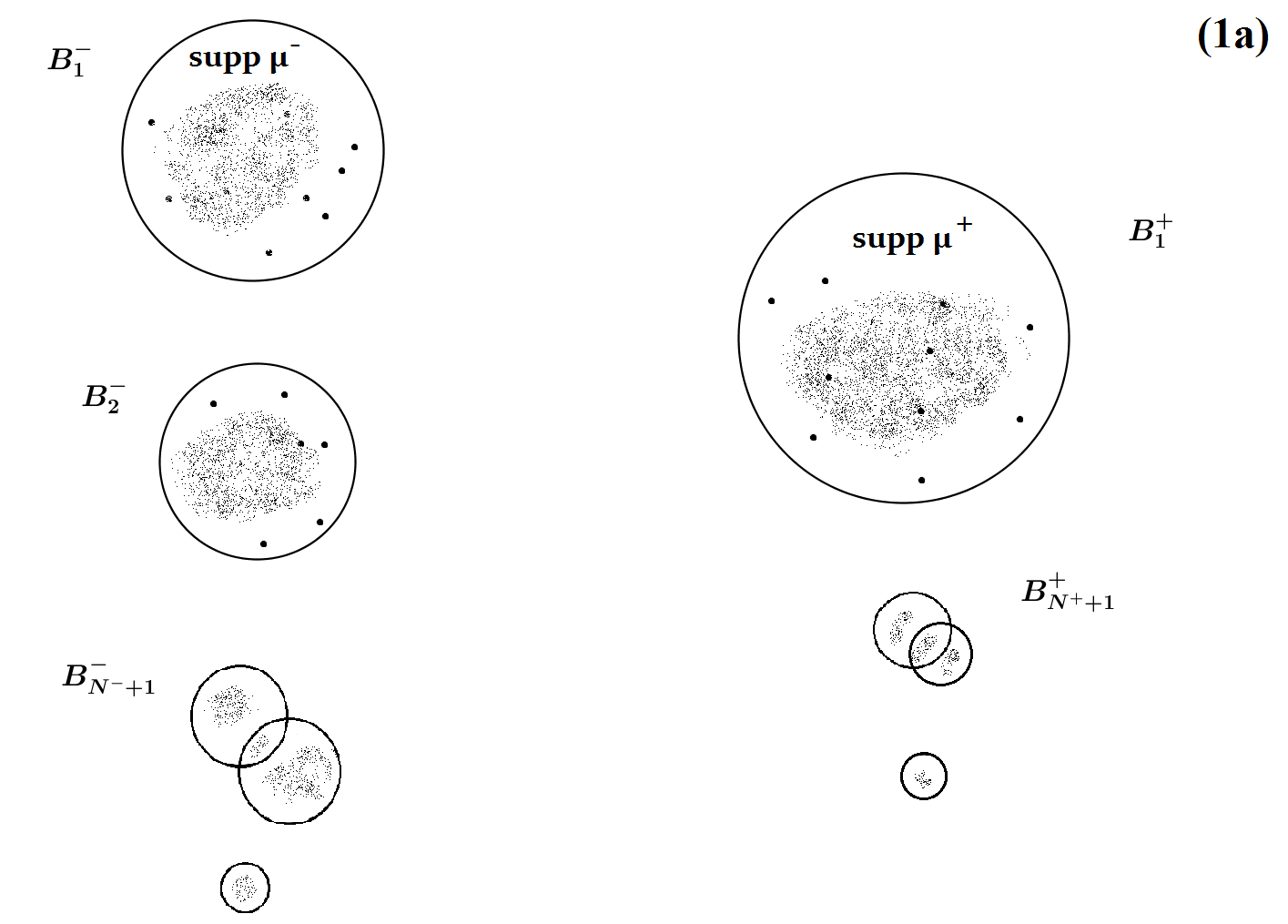} \hspace{1em}
\includegraphics[scale=0.170]{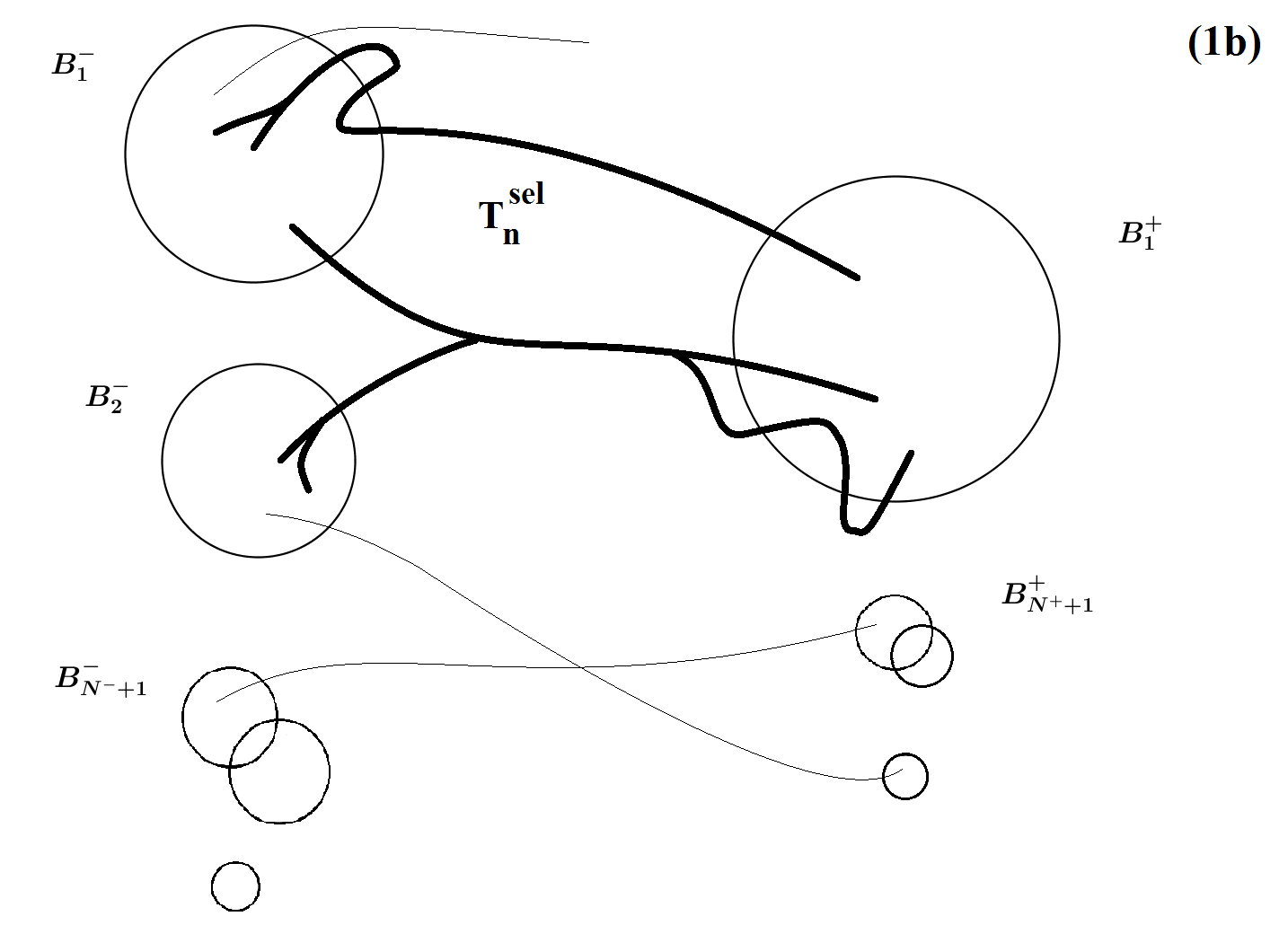}
\caption{Figure (1a) shows the supports of $\mu^+$ and $\mu^-$ and the covering introduced in (1). In Figure (1b) we represented the traffic path $T_n$ and the selection of its curves that begin (respectively end) in the first $N^-$ (respectively $N^+$) balls.}
\label{fig:covering-tn}
\end{center}
\end{figure}

\smallskip

{\it (2) Representation of $T_n$.}  Using Theorem~\ref{s-decompcurr}, we represent each $T_n$ and $T_{opt}$ by a collection of curves weighted by the probability measures $\pi_n$ and $\pi_{opt}$ in $\mathscr{P}(\Lip)$, namely
$$T_n = \int_{\Lip} R_\gamma \, d\pi_n(\gamma), \qquad T_{opt} = \int_{\Lip} R_\gamma \, d\pi_{opt}(\gamma).$$
This representation is essential in order to build an energy competitor for the traffic path $T_n$. 

Intuitively, in the competitor that we want to construct, the mass particles, whose original spatial distribution is represented by $\mu_n^-$, will move for an initial stretch along the curves in the support of $\pi_n$, as long as these curves remain in the balls where they begin. Then, they will be connected to the curves in the support of $\pi_{opt}$ via a ``cheap'' transport supported on the spheres $\partial B^-_i$. Subsequently the particles will move along the curves in the support of $\pi_{opt}$, until they reach the spheres $\partial B^+_i$. From there, another cheap transport supported on the spheres will connect them back to the curves in the support of $\pi_n$ and finally they will be transported to their final destination along the curves of $\pi_n$. Observe that in the process we may have changed the final destination of each single particle, but we preserved the global final particle distribution.

Let us describe now the strategy more in detail. First, we define $\pi^{sel}_n$ as the restriction of $\pi_n$ to curves that start in $\cup_{i=1}^{N^-} B^-_i$ and end in $\cup_{i=1}^{N^+} B^+_i$. We associate to this $\pi_n^{sel}$ a new current $T_n^{sel}$, as represented in Figure~(\ref{fig:covering-tn}b), and we notice that the remaining $\pi_n- \pi_n^{sel}$ carries little mass, by \eqref{eur:fuori-poco} and by the fact that $\partial T_n\rightharpoonup \partial T$.
 We make the further simplifying assumption that
\begin{equation}
\label{eqn:ass-nothing-out}
T_n- T_n^{sel}=0,
\end{equation} even though this is a big simplification since this term cannot be seen as an error in energy.
\smallskip

{\it (3) Construction of a competitor $\tilde T_n^{sel}$ for $T_n^{sel}$.}
We follow the curves representing $T_n^{sel}$ from their starting point, which, by \eqref{eqn:ass-nothing-out} is assumed to be in some $B_i^-$ with $i\in \{1,..., N^-\}$, up to the first time when they touch $\partial B_i^-$. In this way, we define $T^{sel, -}_n$ as in Figure~(\ref{fig:tmeno-toptrestr}a). Similarly, we define  $T^{sel, +}_n$ as the restriction of the curves in  $T^{sel}_n$ from the last time when they touch $\partial B_{i}^+$ up to their final point in $B_{i}^+$ (see again Figure~(\ref{fig:tmeno-toptrestr}a)).

In a similar way, we define $T^{restr}_{opt}$ restricting the curves representing $T_{opt}$ from the first time they exit $\cup_{i=1}^{N^-} B^-_i$ up to the last time they enter $\cup_{i=1}^{N^+} \overline B^+_i$ (see Figure~(\ref{fig:tmeno-toptrestr}b)).

\begin{figure}[h]
\begin{center}
\includegraphics[scale=0.180]{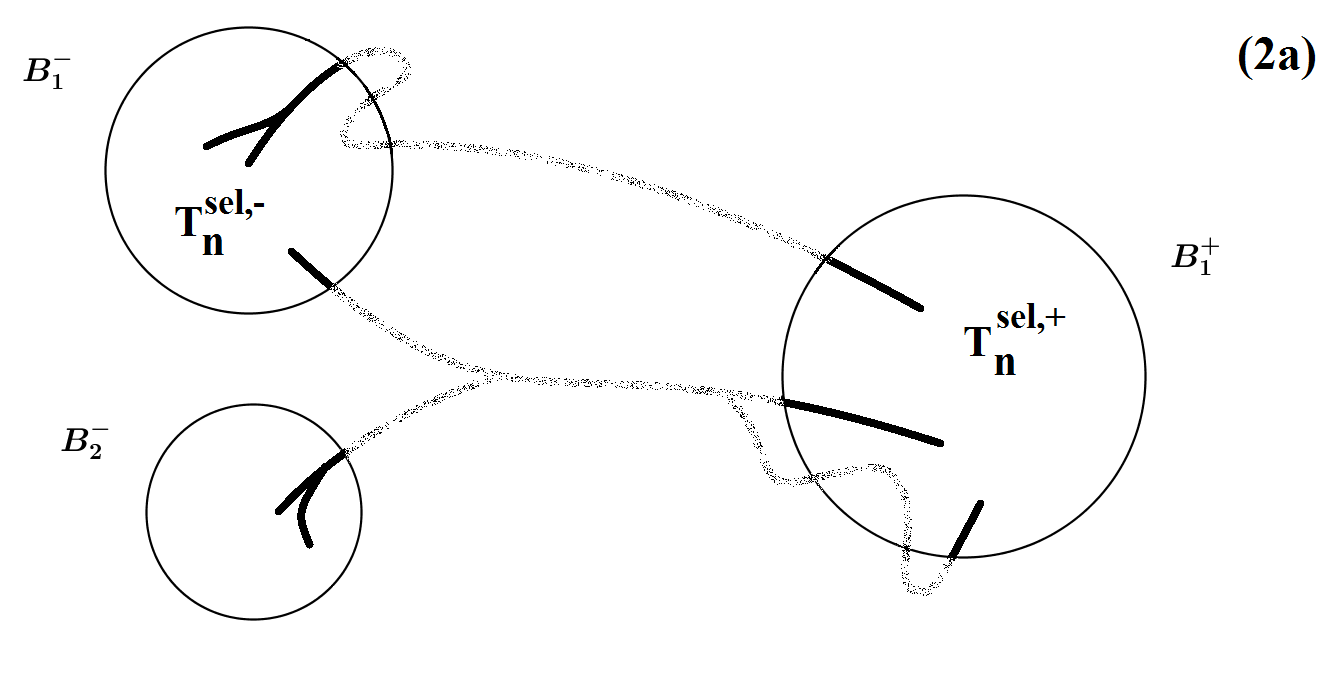} \hspace{1em}
\includegraphics[scale=0.180]{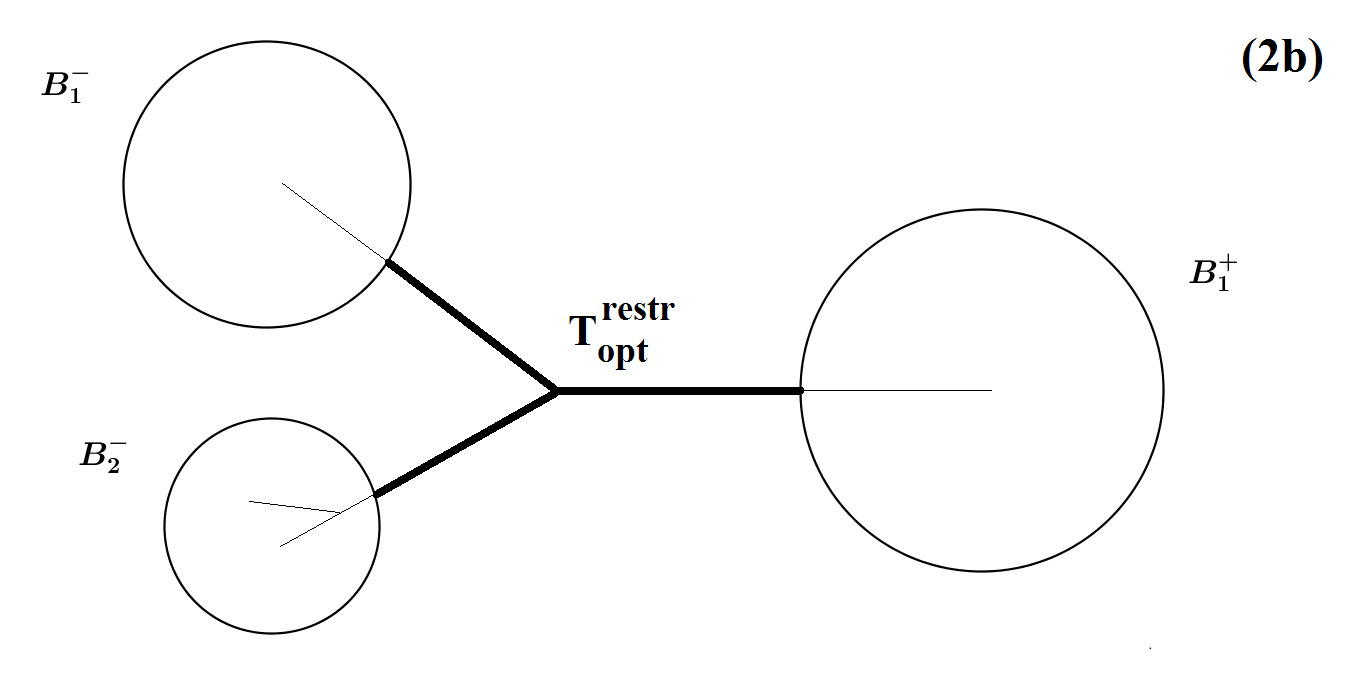}
\caption{In Figure (2a) we mark $T_n^{sel, \pm}$ and in Figure (2b) we mark $T^{restr}_{opt}$.}
\label{fig:tmeno-toptrestr}
\end{center}
\end{figure}

We make the further simplifying assumption that $\mu_n^\pm$ and $\mu^\pm$ have the same quantity of mass in each of the balls $B_{i}^\pm$, $i=1,..., N^\pm$, namely
\begin{equation}
\label{eqn:equal-mass-0}
\mu^\pm_n(B_i^\pm) = \mu^\pm(B_i^\pm) \qquad \mbox{for every } i=1,..., N^\pm,
\end{equation}
or, in other words, that 
\begin{equation}
\label{eqn:equal-mass}
\partial_\pm T^{sel}_n(B_i^\pm) =\partial_\pm T_n(B_i^\pm) = \partial_\pm T_{opt}(B_i^\pm) 
\qquad \mbox{for every } i=1,..., N^\pm.
\end{equation}
We notice that this also implies that
\begin{equation}
\label{eqn:equal-mass-conseq}
\partial_+ T^{sel, -}_n(\partial B_i^-) =\partial_- T^{sel,-}_n(B_i^-)=\partial_- T^{sel}_n(B_i^-)  = \partial_- T_{opt}(B_i^-) =\partial_- T^{restr}_{opt}(\partial B_i^-)
\end{equation}
(and a similar equality holds for $\partial_- T^{sel, +}_n(\partial B_i^+) $).
Indeed, the first equality holds because the traffic path $T^{sel, -}_n$ transports all the mass inside $B_i^-$ on the boundary of $B_i^-$; the last inequality holds because $\pi_{opt}$-a.e. curve exit from $\cup_{i=1}^{N^-} B^-_i$, since it has to end in $\cup_{i=1}^{N^+} \overline B^+_i$ .

Next, we consider a traffic path $T^{conn, - }_n$ that connects  $\partial_+ T^{sel, -}_n$ to $\partial_- T^{restr}_{opt}$ on $\cup_{i} \partial B^-_i$. By \eqref{eqn:equal-mass-conseq}, these two measures can be connected since they have the same mass.  Moreover, by a modification of Theorem \ref{irrigability} (see Lemma \ref{irrigationspher}), the two measures can be connected with finite (and actually small) cost, since they are supported on the union of the $(d-1)$-dimensional spheres $\partial B_i^-$,  and since by assumption in our theorem we required that $\alpha>1-\frac{1}{d-1}$. The cost of this transport is estimated through Lemma~\ref{irrigationspher} by
\begin{equation}
\label{eur:t-conn-}
\MM(T^{conn, - }_n) \leq \sum_{i=1}^{N^-}C_{\alpha,d} r_i^-,
\end{equation}
which is small by \eqref{eqn:chosen-small}.

In a similar way we define a traffic path $T^{conn, +}_n$ that connects  $\partial_- T^{sel, +}_n$ to $\partial_+ T^{restr}_{opt}$ on $\cup_{i} \partial B^+_i$
and enjoys the estimate
\begin{equation}
\label{eur:t-conn+}\MM(T^{conn, + }_n) \leq \sum_{i=1}^{N^+}C_{\alpha,d} r_i^+.
\end{equation}
Finally, we define (see Figure~(\ref{fig:tildetn}))
$$
\tilde T_n^{sel} := T^{sel, -}_n+ T^{conn, -}_n + T^{restr}_{opt} + T^{conn, +}_n+ T^{sel, +}_n.$$

\begin{figure}[h]
\begin{center}
\includegraphics[scale=0.190]{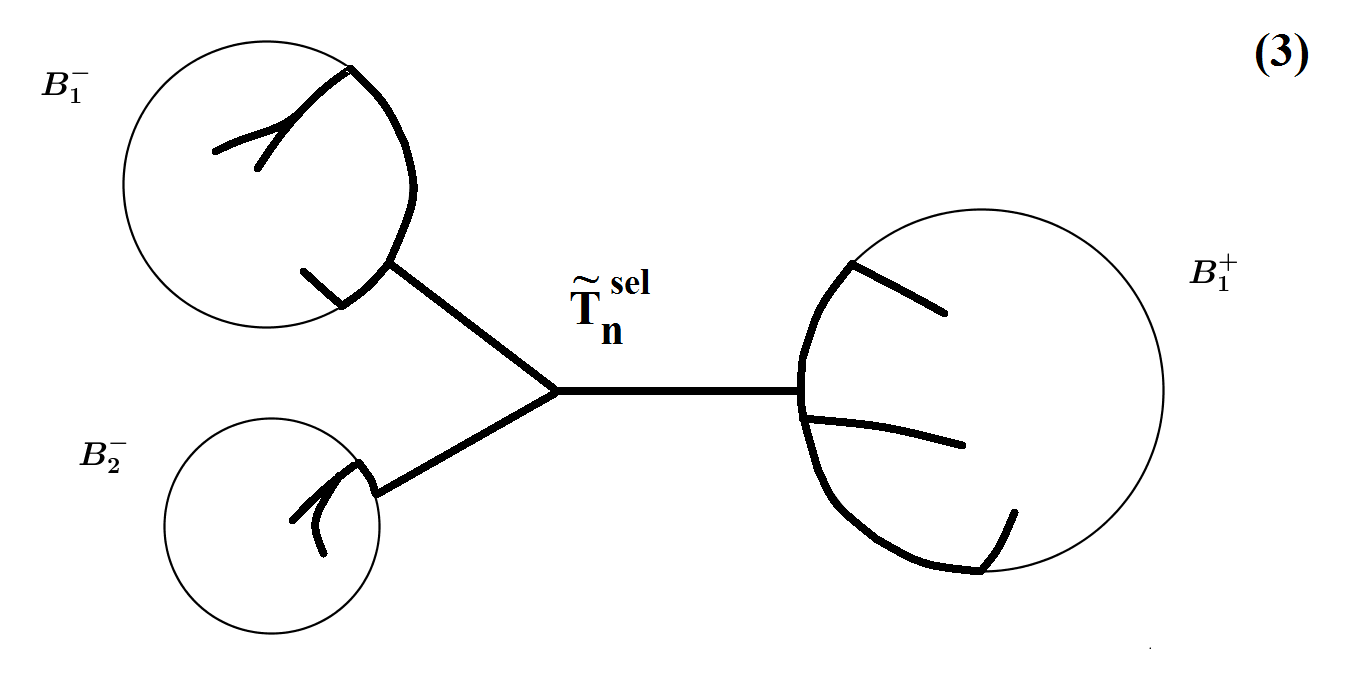}
\caption{The energy competitor $\tilde T_n^{sel}$.}
\label{fig:tildetn}
\end{center}
\end{figure}

\smallskip

{\it (4) Energy estimate for $\tilde T_n^{sel}$ and contradiction.} We show finally that the competitor $\tilde T_n^{sel}$ has strictly less energy than $ T_n$. Since by construction it has the same marginals, then we reach a contradiction. Indeed, by the subadditivity of the $\alpha$-mass, we have
\begin{equation}
\label{eqn:en-est-t-tilde-n-euristics}
\MM(\tilde T_n^{sel}) \leq \MM( T^{sel, -}_n) + \MM(T^{conn, -}_n) + \MM(T^{restr}_{opt}) + \MM(T^{conn, +}_n)+\MM( T^{sel, +}_n)
\end{equation}
By the estimates on the energy of the connections in \eqref{eur:t-conn-} and \eqref{eur:t-conn+} and by the smallness assumptions on the rays, we estimate two terms in the right-hand side of \eqref{eqn:en-est-t-tilde-n-euristics}
\begin{equation}
\label{eqn:en-est-t-tilde-n-euristics2}\MM(T^{conn, -}_n) +  \MM(T^{conn, +}_n) \leq \frac{\Delta}{4}.
\end{equation}
Regarding the first and last terms in the right-hand side of~\eqref{eqn:en-est-t-tilde-n-euristics}, we estimate them with the full energy of $T_n$ inside the balls of the coverings
\begin{equation}
\label{eqn:en-est-t-tilde-n-euristics3}\MM( T^{sel, \pm}_n) \leq  \MM\Big( T_n\trace  \Big( \cup_{i=1}^{N^\pm} \overline B^\pm_i \Big)\Big).
\end{equation}
To bound the energy of $T^{restr}_{opt}$, we first estimate it with the energy of the whole $T_{opt}$. Thanks to the energy gap between $T_{opt}$ and $T$ and \eqref{eqn:chosen-small}, the latter can be estimated choosing  the energy of $T$ inside the coverings below $\Delta/4$:
$$\MM(T^{restr}_{opt}) \leq \MM(T_{opt}) \leq \MM(T) - \Delta \leq \MM\Big(T \trace \Big( \big( \cup_{i=1}^{N^-} \overline B^-_i \big) \cup \big(\cup_{i=1}^{N^+} \overline  B^+_i\big)\Big)^c \Big) -\frac{3\Delta}4
$$ 
By the lower semi-continuity of the $\alpha$-mass on open sets (see Theorem~\ref{lsc}(1))
we deduce that for $n$ large enough
\begin{equation}
\label{eur:toptrestr}
\MM(T^{restr}_{opt})
\leq \MM\Big(T^{sel}_{n} \trace \Big( \big( \cup_{i=1}^{N^-} \overline B^-_i \big) \cup \big(\cup_{i=1}^{N^+} \overline  B^+_i\big)\Big)^c \Big) -\frac{\Delta}2
\end{equation}
Using \eqref{eqn:en-est-t-tilde-n-euristics2}, \eqref{eqn:en-est-t-tilde-n-euristics3}, \eqref{eur:toptrestr} to estimate each term in the right-hand side of \eqref{eqn:en-est-t-tilde-n-euristics} and noticing that the $\alpha$-mass is additive on traffic paths supported on disjoint sets, we find that
\begin{equation*}
\begin{split}
\MM(\tilde T_n^{sel}) \leq & \MM\Big( T^{sel, -}_n\trace  \Big( \cup_{i=1}^{N^-} \overline B^-_i \Big)\Big) + \MM\Big(T^{sel}_{n} \trace \Big( \big( \cup_{i=1}^{N^-} \overline B^-_i \big) \cup \big(\cup_{i=1}^{N^+} \overline  B^+_i\big)\Big)^c \Big) 
\\&+ \MM( T^{sel, +}_n\trace  \Big( \cup_{i=1}^{N^+} \overline B^+_i \Big)) -\frac{\Delta}4 = \MM( T_n^{sel}) -\frac{\Delta}4.
\end{split}
\end{equation*}
This gives a contradiction to the optimality of the energy of $T_n$.

\bigskip
Removing some of the simplifying assumptions that we made in the sketch above is a delicate task and requires new ideas. We briefly describe our strategy. 
\smallskip

In (1), we assumed that the balls $B_i^\pm$ are mutually disjoint. 
If this is not the case, we consider the sets 
$$ C_{i}^\pm:=B_i^\pm \setminus \Big(\cup_{j=1}^{i-1}B_j^\pm \Big)$$ 
as a disjoint cover of the sets $A^\pm$. Then we modify the definition of $T^{sel,-}_n$: for every $=1,..., N^-$, we stop every curve starting in $C^-_i$ as soon as it touches $\partial B^-_i$. The choice to let these curves arrive up to $\partial B^-_i$ (and not only up to $\partial C^-_i$) is related to the fact that $\partial B^-_i$ has a nicer geometry than $\partial C^-_i$ and in particular ensures that the estimate \eqref{eur:t-conn-} holds.
Similarly, we modify $T^{sel,+}_n$.
\smallskip

To remove the assumption $T_n-T^{sel}_n = 0$ in \eqref{eqn:ass-nothing-out}, we consider $\tilde T^{sel}_n + T_n - T^{sel}_n$ as an energy competitor  for $T_n$. To make an energy estimate on this object, we notice first that $T_n-T^{sel}_n $ has small pointwise multiplicity (intensity of flow), since its boundary has small mass and it is made by simple paths (see Proposition \ref{p:propr_good_dec}(2)). Secondly, we prove that the $\alpha$-mass, which in general is sub-additive, is ``almost additive'' between currents which have multiplicities of very different magnitude at every point (Lemma~\ref{lemma:quasiadditive}) and that a suitable lower semi-continuity result holds, involving the restriction of the energy to points with sufficiently large multiplicity (Lemma~\ref{lemma:covering}).

\smallskip
Finally, we need to remove the assumption \eqref{eqn:equal-mass-0} 
: this is another delicate point. Given any $\e>0$, by choosing $n$ large enough, we may assume that
\begin{equation}
\label{eqn:equal-mass-removed}
\partial_\pm T^{sel, \pm}_n(B_i^\pm) \leq (1+\e) \partial_\pm T_{opt}(B_i^\pm).
\end{equation}
Then we use the whole $ (1+\e) T_{opt}$ as a transport outside the balls $\cup_{i=1}^{N^\pm} B_i^\pm$. In view of \eqref{eqn:equal-mass-removed}, this transport might move too much mass from $\cup_{i=1}^{N^-} \partial B_i^-$ to $\cup_{i=1}^{N^+} \partial B_i^+$; however, the amount of mass in excess is small.
Hence, we build another transport with small energy which brings back the mass in excess thanks to Proposition~\ref{l:small-transport}.

\section{Preliminaries for the proof of Theorem~\ref{thm:main}}\label{s:prelim}
\subsection{Restriction of curves to open sets}\label{s:restriction}

Let $A \subseteq \R^d$ be a measurable set. For every $\gamma \in \Lip$, we define the first time $O_A$ in which a curve leaves a set $A$
$$O_A(\gamma) := 
\inf\{ t: \gamma(t) \notin A\},
$$
and the last time $E_A$ in which a curve enters in a set $A$
$$E_A(\gamma) := 
\sup\{ t: \gamma(t) \in A^c\}.
$$
We define the restriction of curves on an interval as a map $res: \Lip \times \{(s,t) \in [0,\infty]^2: s\leq t\} \to \Lip $
\begin{equation}
[res(a,b)(\gamma)](t) = 
\begin{cases}
\gamma(a) \qquad&\mbox{for }t\leq a\\
\gamma(t) \qquad&\mbox{for }a<t<b\\
\gamma(b) \qquad&\mbox{for }t\geq b.
\end{cases}
\end{equation}
In the following, we will often consider the restriction of a curve gamma on a certain set,
or more in general, the restriction of $\gamma$ from an initial time depending on $\gamma$ itself $I(\gamma)$ up to a final time $F(\gamma)$. In this case, we will shorten
$res(I, F)(\gamma):=res(I(\gamma), F(\gamma))(\gamma)$.

The previous definition allows us to state an additional property of good decompositions. 

\begin{proposition}\label{p:propr_good_dec_2}
Let $T \in \mathbf{N}_1(\mathbb{R}^d)$ have a good decomposition $\pi$ as in \eqref{eqn:buona-dec}, and consider two measurable functions $I,F: \Lip \to \R$ with $I \leq F$. Let us assume
 that $\int_{\Lip} \delta_{\gamma(I(\gamma))} d \pi(\gamma)$ and $\int_{\Lip} \delta_{\gamma(F(\gamma))} d \pi(\gamma)$ are mutually singular.
Then the current  
\begin{equation}\label{e:poca_fantasia}
\tilde T:= \int_{\Lip} R_{res(I,F)(\gamma)} d\pi(\gamma)
\end{equation}
has the good decomposition 
$$
\tilde T:= \int_{\Lip} R_{\gamma} d\tilde\pi(\gamma), \qquad \mbox{with } \tilde \pi=({res(I,F)})_\sharp\pi.
$$
 Moreover, if $T= T[E, \tau, \theta]$ is rectifiable, then $\tilde T$ can be written as $\tilde T=T[E,\tau,\tilde \theta]$, with $\tilde \theta \leq \theta$.
\end{proposition}
\begin{remark}\label{remark:A-B-disj}
With the notation of the previous proposition, we notice that the assumptions that $\int_{\Lip} \delta_{\gamma(I(\gamma))} d \pi(\gamma)$ and $\int_{\Lip} \delta_{\gamma(F(\gamma))} d \pi(\gamma)$ are mutually singular in Proposition~\ref{p:propr_good_dec_2} is equivalent to the existence of two disjoint sets $E^-, E^+ \subseteq\R^d$ such that $\gamma(I(\gamma)) \in E^-$ and $\gamma(F(\gamma)) \in E^+$ for $\pi$-a.e. $\gamma$.
\end{remark}
\begin{proof}[Proof of Proposition~\ref{p:propr_good_dec_2}]{\it Proof of the good decomposition property.}
	By Remark \ref{remark:A-B-disj}, it is easy to see that 
	\begin{equation}\label{e:assuno}
	\gamma(I(\gamma)) \neq \gamma(F(\gamma)) \quad\mbox{for }\pi-\mbox{a.e. } \gamma,
	\end{equation}
	and so
	$R_{res(I,F)(\gamma)}$ is a non-constant simple curve, for $\pi$-a.e $\gamma$. Moreover, setting  $T=\tilde T+ T^{resid}$ with
$$T^{resid}: = \int_{\Lip} R_{res(0,I)(\gamma)}+R_{res(F,\infty)(\gamma)} d\pi(\gamma),$$
we have, by the sub-additivity of the mass
\begin{equation}
\begin{split}
\Mass(T) &\leq \Mass(\tilde T)+ \Mass (T^{resid}) 
\\&\leq \int_{\Lip} \big( \Mass(R_{res(0,I)(\gamma)})+ \Mass( R_{res(I, F)(\gamma)})+ \Mass( R_{res(F, \infty)(\gamma)}) \big) d\pi(\gamma)\\
&= \int_{\Lip} \Mass(R_\gamma) d \pi(\gamma),
\end{split}
\end{equation}
where in the last line we use that $\pi$-a.e. curve $\gamma$ is simple. Since, by \eqref{eqn:buona-dec-mass-T}, equality holds between the first and  the last term, every inequality should be an equality and in particular
$$\Mass(\tilde T)=
\int_{\Lip} \Mass(R_{res(I,F)(\gamma)}) d\pi(\gamma) = \int_{\Lip} \Mass(R_\gamma) d\tilde \pi(\gamma).
$$
In order to obtain the same equality for $\partial \tilde T$, 
we first notice that, by \eqref{buona-dec-boundary}, it holds
$$
\partial \tilde T =  \int_{\Lip} \partial R_{res(I,F)(\gamma)} d \pi(\gamma) = \int_{\Lip} \big( \delta_{\gamma(F(\gamma))} - \delta_{\gamma(I(\gamma))} \big) d \pi(\gamma).
$$
By assumption, the measures $\int_{\Lip} \delta_{\gamma(I(\gamma))} d \pi(\gamma)$ and $\int_{\Lip} \delta_{\gamma(F(\gamma))} d \pi(\gamma)$ are mutually singular. Hence, 
$$\partial_-\tilde T=\int_{\Lip} \delta_{\gamma(I(\gamma))} d \pi(\gamma)\quad \mbox{ and }\quad \partial_+\tilde T=\int_{\Lip} \delta_{\gamma(F(\gamma))} d \pi(\gamma),$$
which yields, by \eqref{e:massbord},
$$\Mass(\partial\tilde T)=2\Mass(\partial_-\tilde T)= 2\pi(\Lip)=\int_\Lip\Mass(\partial R_\gamma) d\tilde\pi(\gamma).$$
This concludes the proof that \eqref{e:poca_fantasia} is a good decomposition. 

\smallskip
{\it Proof of the estimate on the multiplicity.}
By the good decomposition property proved above and the formula \eqref{eqn:dens-acycl}, we get that for $\Haus^1$-a.e. $x\in E$
\begin{equation}
\begin{split}
0 \leq \tilde \theta(x)&= \tilde \pi(\{\gamma: x \in {\rm{Im}}(\gamma) \})= \pi(\{\gamma: x \in {\rm{Im}}({res(I,F)(\gamma)}) \})\\
&\leq \pi(\{\gamma: x \in {\rm{Im}}(\gamma) \})=\theta(x),
\end{split}
\end{equation}
where in the inequality we used that if $x \in {\rm{Im}}({res(I,F)(\gamma)})$ then $x \in {\rm{Im}}(\gamma)$.
This concludes the proof of the claim.
\end{proof}

\subsection{Dimension reduction}
The next lemma is a fundamental tool for the proof of our main result. Indeed it allows us to transport measures which are supported on $(d-1)$-dimensional spheres, decreasing the critical threshold for which we have quantitative upper bounds on the minimal transport energy (see Theorem~\ref{irrigability}). Its proof is a simple combination of Theorem \ref{irrigability} and Proposition \ref{p:push_forw}.

\begin{lemma}\label{irrigationspher}
Let $\alpha > 1-\frac 1{d-1}$. Given two measures $\mu^-$ and $\mu^+$ with mass $M$ in $\R^d$ supported on $\partial B(x,r)$, there exists a current $T\in \TP(\mu^-,\mu^+)$ such that
$$\Mass^\alpha(T)\leq C_{\alpha,d}M^\alpha r,$$
where $C_{\alpha,d}$ is a constant depending only on $\alpha$ and $d$.
\end{lemma}
\begin{proof}
In this proof we denote by $B^{d-1}(0,r)$ the open ball in $\R^{d-1}$ centred at 0 with radius $r$. Let $p\in \partial B(x,r)$ such that $\mu^\pm(\{p\})=0$. It is easy to see that there exists a constant $C:= C(d)$ and a $1$-Lipschitz function $f:\overline{B^{d-1}(0,Cr)}\to \partial B(x,r)\subset \R^d$ which ``wraps'' $B^{d-1}(0,Cr)$ onto $\partial B(x,r) \setminus \{p\}$. More precisely, we can require that
$$f^{-1}(\{p\})=\partial B^{d-1}(0,Cr) \qquad \mbox{and} \qquad f \mbox{ is injective on }B^{d-1}(0,Cr).$$
Let $\nu^\pm:=[(f|_{\partial B(x,r) \setminus \{p\} } )^{-1}]_\sharp\mu^{\pm}$ and observe that $f_\sharp  \nu ^\pm= \mu^\pm$. By Theorem \ref{irrigability}, there exists $S\in \TP(\nu^-,\nu^+)$ with $\Mass^\alpha(S)\leq C_{\alpha,d}M^\alpha 2r$.
 
We observe that $T:=f_\sharp S$ belongs to $\TP(\mu^-,\mu^+)$, indeed
$$\partial (f_\sharp S)=f_\sharp \partial S=f_\sharp ((f|_{\partial B(x,r) \setminus \{p\} } )^{-1}_\sharp\mu^{+}-(f|_{\partial B(x,r) \setminus \{p\} } )^{-1}_\sharp\mu^{-})=\mu^{+}-\mu^{-}=\partial S,$$
and trivially $T$ is supported on $\partial B(x,r)$. The estimate on the $\alpha$-mass of $T$ follows immediately from Proposition \ref{p:push_forw}.
\end{proof}

\subsection{Covering results}
In this subsection we prove two elementary covering results. Referring to the notation introduced in Section \ref{s:ideas}, Lemma \ref{lemma:covering-other} allows us to cover the sets $A^\pm$ with balls satisfying \eqref{eqn:chosen-small} such that for every $n\in\N$ almost no curve in the representation of $T_n$ begins or ends on the corresponding spheres. With Lemma \ref{lemma:covering} we want to guarantee that it is possible to cover the sets $\supp(\mu^-)$ and  $\supp(\mu^+)$, which by assumption are disjoint, with two disjoint families of small balls. This time we do not require any smallness assumption on the sum of the radii, but we want to control the number of balls in each family.

\begin{lemma}\label{lemma:covering-other}
	Consider a family of $1$-currents $T,T', (T_n)_{n \in \N}\in \mathbf{N}_1(\mathbb{R}^d)\cap \mathbf{R}_1(\mathbb{R}^d)$, such that $\MM(T), \MM(T')<+\infty$ and $\partial_\pm T=\partial_\pm T'$. Given a set $A$ such that $\Haus^1(A)=0$, and $\e>0$, there exists a covering of $A$ with open balls $(B(x_{i},r_{i}))_{i\in \N}$
	such that
	$$\partial_\pm T(\partial B(x_{i},r_{i}))= \partial_\pm T_n(\partial B(x_{i},r_{i}))=0 \qquad \mbox{for every }i,n \in \N,$$
	\begin{equation}\label{eqn:ts-cov}
		\MM(T\trace \bigcup_{i \in \N} \overline{ B(x_{i},r_{i}) })<\e\qquad \mbox{ and } \qquad \MM(T'\trace \bigcup_{i \in \N} \overline{ B(x_{i},r_{i}) })<\e,
	\end{equation}
	and
\begin{equation}\label{eqn:ts-cov1}	
	\sum_{i=1}^\infty r_{i}<\e.
	\end{equation}
\end{lemma}
\begin{proof}
	We define on $\R^d$ the finite measure $\nu$ by
	\[
	\nu(E)=\MM(T\trace E)+\MM(T'\trace E)\quad \text{for every Borel set $E$}
	\] 
	and we observe that $\nu$ vanishes on $\Haus^1$-null sets.
	
	Since $\Haus^1(A)=0$, for every $j \in \N$ we can find a covering of $A$ with balls $\{B(x^{(j)}_{i},r^{(j)}_{i})\}_{i \in \N}$ such that 
	$$\sum_{i \in \N} r^{(j)}_{i} < \frac{1}{2^{j+1}},$$
	and moreover, since for every point $x$ there are only countably many radii $r$ such that $\partial_\pm T(\partial B(x,r))\neq 0$ or $\partial_\pm T_n(\partial B(x,r))\neq 0$ for some $n$, then we can also assume (possibly enlarging slightly the previous radii) that	
	$$ \partial_\pm T(\partial B(x^{(j)}_{i},r^{(j)}_{i}))=\partial_\pm T_n(\partial B(x^{(j)}_{i},r^{(j)}_{i}))= 0 \qquad \mbox{for every }i,n \in \N.$$
We define 
	$$A^{(j)} = \bigcup_{i \in \N} \overline{ B(x^{(j)}_{i},2r^{(j)}_{i})}.$$
We consider the decreasing sequence of sets and their intersection
	$$(B^{(j)})_{j\in\N}:= \bigcup_{j' \geq j} A^{(j')}, \qquad B = \bigcap_{j \in \N} B^{(j)}.$$
	We notice that $A^{(j)} \subseteq B^{(j)}$ for every $j \in \N$
	and that $\Haus^1 (B) = 0$, because $B$ can be covered with each $B^{(j)}$, which in turn is made by balls whose radii satisfy the estimate
	$$\sum_{j'\geq j} \sum_{i \in \N} r^{(j')}_{i} < \sum_{j'\geq j}\frac{1}{2^{j'+1}} =\frac{1}{2^{j}}.$$
	We consequently have on the decreasing sequence of sets $(B^{(j)})_{j\in\N}$:
	$$\lim_{j \to \infty} \nu(B^{(j)})= \nu(\cap_j B^{(j)}) = \nu (B)
	$$
	and we conclude that $\nu(B)=0$ and that
	$$\lim_{j \to \infty} \nu(A^{(j)}) \leq \lim_{j \to \infty} \nu(B^{(j)})=0
	.$$
	Therefore, choosing $j$ large enough, the covering $(B(x^{(j)}_{i},r^{(j)}_{i}))_{i \in \N}$ satisfies the conditions in \eqref{eqn:ts-cov}, \eqref{eqn:ts-cov1}.
\end{proof}

\begin{lemma}\label{lemma:covering}
Given $r>0$ and $K \subset X$, with  $K$ compact.
There exists a finite number $M:=M(X,r)$ and a family of balls $\{B(x_i,r_i)\}_{i=1}^{M}$, covering $K$, such that
$$r_i< \frac r3, \qquad x_i \in K. $$
\end{lemma}

\begin{proof} 
We cover $K$ with balls $B(x,r/4)$, $x\in K$ and, by Vitali's covering theorem, we can extract a finite sub-covering, indexed by $\{1,..., M\}$ such that the balls $\{B(x_j,r/20)\}_{j=1,...,M}$ are disjoint and the balls $\{B(x_j,r/4)\}_{j=1,...,M}$ cover $K$.        
By the disjointness of $\{B(x_j,r/20)\}_{j=1,...,M}$ and since these balls are all contained in $U_r(X):=\{y\in\R^d:{\rm{dist}}(y,X)<r\}$, it follows that
$$M |B(0,r/20)| \leq |U_r(X)|,$$
which completes the proof of the lemma.
\end{proof}

\subsection{A semi-continuity and a quasi-additivity result}
In this subsection we collect two results which allow us to get rid of the simplifying assumption \eqref{eqn:ass-nothing-out} in the sketch of Section \ref{s:ideas}. Lemma \ref{high_multiplicity} improves Theorem \ref{lsc}(1), allowing us to consider in the right hand side of the inequality \eqref{e:lsc_res} only the portion of the currents $T_n$ which have sufficiently high multiplicity. Lemma \ref{lemma:quasiadditive} states that the $\alpha$-mass is ``quasi-additive'' if the two addenda have multiplicities of different orders of magnitude.

\begin{lemma}\label{high_multiplicity}
	Let $C>0$, $A \subseteq \R^d$ an open set, and let $T'= T[ E', \tau', \theta' ] \in \mathbf{R}_{1}(\R^d)$ and $T:= T[E,\tau, \theta] \in \mathbf{R}_{1}(\R^d)$  be rectifiable $1$-currents with 
	\begin{equation}
	\label{hp:mass-currents-bounded}
	\mathbb M^\alpha(T'), \MM(T) \leq C.
	\end{equation}
	Then, 
	for every $\e>0$ there exists $\delta:= \delta(d, \alpha, \e, C, A, T)>0$ (independent of $T'$) such that, if $\Flat (T- T') \leq \delta$,
	\begin{equation}\label{e:high_mult}
	\mathbb M^\alpha(T'\trace\{x\in A: \theta'(x)>\delta\})\geq \mathbb M^\alpha(T \trace A)-\e.
	\end{equation}
	
\end{lemma}
\begin{proof}
	For every $\delta>0$, by \eqref{hp:mass-currents-bounded} it holds
	\begin{equation}
	\label{eqn:mass-to-0}
	\mathbb M(T'\trace\{\theta'\leq \delta \})<\delta^{1-\alpha}\mathbb M^\alpha(T'\trace\{\theta'\leq \delta\})<\delta^{1-\alpha}C.
	\end{equation}
	Hence, 
	\begin{equation}
	\begin{split}
	\mathbb{F}(T-T'\trace\{\theta'>\delta\})
	&\leq\mathbb{F}(T-T')+\mathbb{F}(T'- T'\trace\{\theta'>\delta \}) 
	\\&=\mathbb{F}(T-T')+\mathbb{F}(T'\trace\{\theta'\leq \delta \}) 
	\\ & \leq \mathbb{F}(T-T')+ \mathbb M(T'\trace\{\theta'\leq\delta \})
	\\&\leq \mathbb{F}(T-T')+ C \delta^{1-\alpha} \leq \delta+ C \delta^{1-\alpha} .
	\end{split}
	\end{equation}
	By the lower semi-continuity of the $\alpha$-mass with respect to the flat convergence (as stated in Theorem~\ref{lsc}(1)), there exists $\delta_0:= \delta_0(d,\alpha,\e, A,T)$ such that for any rectifiable $1$-current $\tilde T$ satisfying $\mathbb F(\tilde T - T) \leq \delta_0$ we have $\mathbb M^\alpha(\tilde T \trace A ) \geq \mathbb M^\alpha(T \trace A) - \e$. We conclude the proof choosing $\delta$ sufficiently small so that $\delta+C \delta^{1-\alpha} \leq \delta_0$.
\end{proof}

\begin{lemma}\label{lemma:quasiadditive}
Let $\varepsilon \in (0,1/4)$, $T_1=T[E_1,\tau_1,\theta_1], T_2=T[E_2,\tau_2,\theta_2]\in \mathbf{R}_{1}(\R^d)$ be rectifiable $1$-currents with $\theta_1<\e \theta_2$, $\Haus^1$-a.e. on $E_1 \cap E_2$. Then
\begin{equation}\label{e:lemma_quasiadd}
(1+4 \e^{\alpha})\Mass^\alpha(T_1+T_2)\geq \Mass^\alpha(T_1)+\Mass^\alpha(T_2).
\end{equation}
\end{lemma}
\begin{proof} Firstly we observe that on $E_1\cap E_2$ we have
\begin{equation}\label{spezza}
2\e(\theta_2-\theta_1)\geq\theta_1;\quad (1+2\e)(\theta_2-\theta_1)\geq \theta_2.
\end{equation}
Now we compute
\begin{equation*}
\begin{split}
(1+&4 \e^{\alpha})\Mass^\alpha(T_1+T_2)=(1+4 \e^{\alpha})\Mass^\alpha(T_1\trace(E_1\setminus E_2))\\
&\quad +(1+4 \e^{\alpha})\Mass^\alpha(T_2\trace (E_2\setminus E_1)) +(1+4 \e^{\alpha})\Mass^\alpha((T_1+T_2)\trace(E_1\cap E_2))\\
&\geq \Mass^\alpha(T_1\trace(E_1\setminus E_2))+\Mass^\alpha(T_2\trace (E_2\setminus E_1))\\
&\quad  + \big((2\e)^\alpha + (1+2\e)^\alpha \big)\Mass^\alpha((T_1+T_2)\trace(E_1\cap E_2))
.
\end{split}
\end{equation*}
We estimate the last term thanks to \eqref{spezza} to get
\begin{equation*}
\begin{split}
 \big(&(2\e)^\alpha +(1+2\e)^\alpha \big)\Mass^\alpha((T_1+T_2)\trace(E_1\cap E_2))\\
&\qquad {\geq} \Mass^\alpha(T_1\trace(E_1\cap E_2)) + \Mass^\alpha(T_2\trace(E_1\cap E_2 ))
.
\end{split}
\end{equation*}Putting together the previous two inequalities, we get \eqref{e:lemma_quasiadd}.
\end{proof}


\subsection{Absolute continuity of the transportation cost}
The next proposition is the fundamental tool to get rid of the simplifying assumption \eqref{eqn:equal-mass} in the sketch of Section \ref{s:ideas}. It ensures that if there exists a traffic path of finite cost transporting a measure $\mu^-$ onto a measure $\mu^+$, then a transportation between two ``small'' sub-measures of $\mu^-$ and $\mu^+$ of equal mass is cheap.

\begin{proposition}\label{l:small-transport}
Let $\mu^-,\mu^+ \in \M(X)$, be non-trivial measures with $\mu^-(X)=\mu^+(X)<\infty$. Assume 
$$\supp(\mu^-)\cap \supp(\mu^+)=\emptyset,$$
with $\MM ( \mu^-,\mu^+)<\infty$.
Then for every $\e>0$ there exists $\delta>0$ such that for every pair of measures $\nu^-\leq \mu^-$ and $\nu^+\leq \mu^+$ verifying
$$\nu^-(X)=\nu^+(X)\leq \delta,$$
then $\MM( \nu^-,\nu^+)\leq \e$.
\end{proposition}

\begin{proof}
Without loss of generality, we may assume $\mu^-,\mu^+\in\mathscr{P}(X)$. By assumption, there exists $T\in \OTP(\mu^-,\mu^+)$, such that $\MM (T)<+\infty$.

Let $T = \int _{\Lip} R_\gamma \, d\pi (\gamma)$ be a good decomposition of $T$ and define the finite measure $\pi^\pm \in \M(\Lip)$, prescribing their Radon--Nikod\'{y}m densities w.r.t. $\pi$, as 
$$d\pi^-(\gamma):= \frac{d\nu^-}{d \mu^-} (\gamma(0))d\pi (\gamma), \qquad d\pi^+(\gamma):= \frac{d\nu^+}{d \mu^+} (\gamma(\infty))d\pi (\gamma).$$
We denote
\begin{equation}\label{gooddec}
T^\pm=T[E^\pm,\tau^\pm,\theta^\pm]:=  \int_{\Lip} R_\gamma d\pi^\pm(\gamma).
\end{equation}
Let us consider $\delta>0$ as fixed. For the moment we only require that $\delta<\delta_0:= \delta(\e/4)$ of Proposition~\ref{p:propr_good_dec}(4). Further restrictions will be given later.
Since $\pi^\pm(\Lip) = \nu^\pm(X)\leq \delta$, from Proposition~\ref{p:propr_good_dec} (3) and (4)  we deduce that the decompositions in \eqref{gooddec} are good and that
\begin{equation}\label{smallmass}
\MM (T^\pm) \leq \frac{\e}{4}.
\end{equation}
By \eqref{buona-dec-boundary} we can write the boundaries of $T^\pm$ in terms of the decomposition as
\begin{equation}
\label{eqn:bd-t-pm}
\partial_- T^\pm = \int_{\Lip}\delta_{\gamma(0)} d \pi^\pm (\gamma)
\qquad \mbox{and} \qquad
\partial_+ T^\pm = \int_{\Lip}\delta_{\gamma(\infty)} d \pi^\pm (\gamma).
\end{equation}
We apply Lemma~\ref{lemma:covering} twice to $K:=\supp (\mu ^\pm)$ and $r:=\frac 13$dist$(\supp (\mu ^-),\supp (\mu ^+))$ to find a finite covering of $\supp (\mu^\pm)$ made by at most $M(X,r)$ open balls
$$
B^\pm_i:=B(x_i^\pm,r_i^\pm) \qquad i=1,..., M^\pm.
$$
For every $i=1,..., M^\pm$ let us define 
$$C^\pm:=\bigcup_{i}B^\pm_i.$$
By the choice of $r$, the sets $C^+$ and $C^-$ are disjoint. Hence, since $\supp(\partial_\pm T) \subseteq C^\pm$ and since \eqref{eqn:bd-t-pm} is in force, then $\pi^\pm$-a.e. $\gamma \in \Lip$ verifies
\begin{equation}
\label{eqn:curve-fuori-dalle-palle}\gamma(0)\in C^- \qquad \mbox{and} \qquad \gamma(\infty)\in C^+.
\end{equation}
We define the rectifiable $1$-currents
\begin{equation}\label{gooddec1}
\begin{split}
 T^{cut,-}=T[E^{cut,-},\tau^{cut,-},\theta^{cut,-}]:=\int_{\Lip} R_{res(0,O_{C^-})(\gamma)} d\pi^-(\gamma),\\ T^{cut,+}=T[E^{cut,+},\tau^{cut,+},\theta^{cut,+}]:=\int_{\Lip} R_{res(E_{C^+},\infty)(\gamma)} d\pi^+(\gamma).
 \end{split}
 \end{equation}
By Proposition~\ref{p:propr_good_dec_2}, \eqref{gooddec1} are good decompositions. {Here we use a little abuse of notation, since the good decomposition of $T^{cut,-}$ would be the push-forward measure 
$$
\big(res(0,O_{C^-})(\cdot)\big)_\sharp \pi^-$$ and similarly for $T^{cut,+}$}. In particular, by point (1) of Proposition~\ref{p:propr_good_dec} it holds 
\begin{equation}
\label{eqn:boundary-t-meno-senza-mu}\partial_- T^{cut,-} =  \int_{\Lip} \delta_{\gamma(0)} d \pi^-(\gamma),
\qquad
\partial_+ T^{cut,-} =  \int_{\Lip} \delta_{\gamma(O_{C^-})} d \pi^-(\gamma)
\end{equation}
  Hence we deduce 
$$\supp(\partial_+ T^{cut,-})\subseteq \partial C^-  \qquad \mbox{and} \qquad \supp(\partial_- T^{cut,+})\subseteq \partial C^+.$$
By the good decomposition property of $T^{cut,-}$ and of $T^-$ and by Proposition~\ref{p:propr_good_dec_2} for $\Haus^1$-a.e. $x\in E^-\cap E^{cut,-}$ we have that
\begin{equation}
\begin{split}
\theta^{cut,-}(x) 
\leq \theta^{-}(x).
\end{split}
\end{equation}
Thanks to \eqref{smallmass}, we deduce that $T^{cut,\pm}$ have small energy
\begin{equation}
\label{eqn:small-energy-cut}
\MM (T^{cut,-}) = \int_{E^{cut,-}}(\theta^{cut,-})^{\alpha} d\Haus^1 \leq \int_{E^{-}}(\theta^-)^{\alpha} d\Haus^1 =  \MM (T^-)
\leq \frac{\varepsilon}{4}.
\end{equation}
With similar computations we can prove the same energy estimate for  $ T^{cut,+}$.

\bigskip

Let $\{y_1^-,..., y_{M^-}^-\}_{i=1,..., M^-} \subseteq \R^d $ and $\{y_1^+,..., y_{M^+}^+\}_{i=1,..., M^+} \subseteq \R^d $ be two sets of distinct points such that $y^\pm_i \in \partial B_i^\pm$ for every $i= 1,..., M^\pm$. For every $i=1,..., M^-$ we define the weight $w_i^\pm \in (0,\infty)$ as 
$$
w_i^- := (\partial_+ T^{cut,-}) \Big( \partial B_i^- \setminus \bigcup_{j=1}^{i-1}\partial B_j^- \Big)
$$
and
$$
w_i^+ := (\partial_- T^{cut,+}) \Big( \partial B_i^+ \setminus \bigcup_{j=1}^{i-1}\partial B_j^+ \Big).
$$
We consider the measures $\sigma^\pm:=\sum_{i=1}^{M^\pm} w^\pm_i \delta_{y^\pm_i} $, whose total mass is equal to $\nu^\pm(X)\leq \delta$. Indeed we proved in \eqref{eqn:boundary-t-meno-senza-mu}, that $\partial_- T^{cut,-}=\partial_- T_-$ and consequently
$$\sigma^-(X)=\partial_+T^{cut,-}(X)=\partial_- T^{cut,-}(X)=\partial_- T_-(X)=\nu^-(X)\leq\delta$$
and analogously 
$$\sigma^+(X)=\partial_-T^{cut,+}(X)=\partial_+ T^{cut,+}(X)=\partial_+ T(X)=\nu^+(X)\leq\delta.$$
We claim that there exists $T^{conn, - } \in \TP( \partial_+ T^{cut,-} , \sigma^- )$ with
$$\MM(T^{conn, - } ) \leq C(d, \alpha, X, r)\delta.$$
Similarly, we claim that  there exists $T^{conn,+ } \in \TP( \partial_- T^{cut,+} , \sigma^+ )$ with
$$\MM(T^{conn, + } ) \leq C(d, \alpha, X, r)\delta.$$
Indeed let us consider for every $i=1,..., M^-$ an optimal traffic path 
$$T^{conn, - }_{i} \in \OTP \big(  (\partial_+ T^{cut,-}) \trace\big( \partial B_i^- \setminus \cup_{j=1}^{i-1} \partial B_j^- \big), w_i \delta_{y_i^-}\big)$$
and observe that, by Lemma \ref{irrigationspher} 
$$\MM(T^{conn, - }_{i} ) \leq C(d, \alpha)\delta r.$$
If we consider now
$$T^{conn, - }:= \sum_{i=1}^{M^-} T^{conn, - }_{i},$$
we notice that $T^{conn, - }\in \TP(  \partial_+ T^{cut,-} ,\sigma^-) $ and by the sub-additivity of the $\alpha$-mass \eqref{eqn:mass-subadd} we obtain that 
$$\MM(T^{conn, - } ) \leq \sum_{i=1}^{M^-} \MM(T^{conn, - }_i ) \leq M^-(X,r) C(d, \alpha)\delta r \leq  C(d, \alpha, X,r)\delta$$
and this proves the claim.

\bigskip
Finally we observe that there exists $T^{graph} \in \TP \big(\sigma^- ,  \sigma^+\big) $ with 
$$
\MM(T^{graph} ) \leq \delta^\alpha C(d,X).
$$
The simplest way to find such traffic path is to connect all the points in the support of $\sigma^\pm$ to a fixed point in $X$. The estimate of its $\alpha$-mass is trivial.
Overall, we find that 
$$T^{new} := T^{cut,-} + T^{conn, - } +T^{graph} + T^{conn, +}+ T^{cut,+} \in \TP(\nu^-,\nu^+)$$
 and its energy is estimated using  the sub-additivity \eqref{eqn:mass-subadd} and the previous estimates (observing that $\delta\leq \delta^\alpha$ for $\delta\leq 1$)
$$\MM(T^{new}) \leq C(d, \alpha, X,r)\delta^\alpha + \frac{\e}2.$$
By choosing $\delta$ sufficiently small, we obtain that the last quantity is less than or equal to $\e$. This concludes the proof of the lemma.



\end{proof}

\begin{corollary}
Let $\mu^-,\mu^+ \in \mathscr P(X)$. Assume 
$$\supp(\mu^-)\cap \supp(\mu^+)=\emptyset,$$
with $\MM ( \mu^-,\mu^+)<\infty$. Then for every pair of sequences $(\mu^-_n)_{n\in \N}$ and $(\mu^+_n)_{n\in \N}$ with $\mu^-_n(\R^d) = \mu^+_n(\R^d)$, $\mu^-_n \leq \mu^-$, $\mu^+_n \leq \mu^+$ for every $n\in \N$ and with $$\lim_{n\to \infty} \mu^-(X)-\mu^-_n(X) = 0,$$ we have that
$$\lim_{n\to \infty} \Mass^\alpha(\mu^-_n, \mu^+_n) =\Mass^\alpha(\mu^-, \mu^+).$$
\end{corollary}
\begin{proof}
By the lower semi-continuity of the $\alpha$-mass (Theorem \ref{lsc}(1)), we only need to show that
\begin{equation}
\label{eqn:upper-manca}
\limsup_{n\to \infty} \Mass^\alpha(\mu^-_n, \mu^+_n) \leq \Mass^\alpha(\mu^-, \mu^+).
\end{equation}
Indeed, if we assume \eqref{eqn:upper-manca}, by Corollary \ref{existence}, and by the compactness of normal currents (see \cite[4.2.17(1)]{FedererBOOK}) we can consider a sequence of optimizers $(T_{n_k})_{k\in\N}$, where $T_{n_k}\in\OTP(\mu^-_{n_k},\mu^+_{n_k})$ converge to a traffic path $T\in\TP(\mu^-,\mu^+)$ with finite cost and
$$\lim_{k\to \infty}\MM(T_{n_k})=\liminf_{n\to \infty}\MM(T_n).$$ 
Hence we compute
$$\Mass^\alpha(\mu^-, \mu^+)\leq\MM(T)\overset{\eqref{e:lsc_res}}{\leq}\liminf_{k\to \infty}\MM(T_{n_k})=\liminf_{n\to\infty}\MM(T_n).$$
In order to prove \eqref{eqn:upper-manca}, we let $T \in \OTP(\mu^-, \mu^+)$. Since by assumption the measures
$\mu^--\mu^-_n$ and $\mu^+- \mu^+_n$ are non-negative, are converging to $0$ and, for each fixed $n$, they have the same mass, we deduce by point (4) of Proposition~\ref{p:propr_good_dec} that, denoting by $T_n'$ any optimal path in $\OTP(\mu^--\mu^-_n, \mu^+-\mu^+_n)$, 
$$\lim_{n\to \infty}\Mass^\alpha(T_n') =0.$$ 
Let $T_n = T- T_n' \in \TP(\mu^-_n, \mu^+_n)$. By the sub-additivity of the $\alpha$-mass \eqref{eqn:mass-subadd}
$$\Mass^\alpha(T_n) \leq \Mass^\alpha(T) + \Mass^\alpha(T_n').$$ 
Letting $n\to \infty$ we obtain \eqref{eqn:upper-manca}.
\end{proof}
\begin{remark}
From this observation the stability follows as in the case $\alpha>1-1/d$ as soon as the approximating sequences are sub-measures of $\mu^-$ and $\mu^+$ respectively. In particular, if $\mu^-$ is a Dirac delta and $\mu^+$ is an atomic measure, then an optimal traffic path connecting $\mu^-$ to $\mu^+$ can be obtained as the limit of the optimal traffic paths connecting the correct ``rescaled'' measure of $\mu^-$ to the discrete measure obtained restricting $\mu^+$ to suitable sets of finitely many points.
\end{remark}


\section{Proof of Theorem~\ref{thm:main}}\label{s:proof}
Up to rescaling, we can assume that $\mu^-$ and $\mu^+$ are probability measures. Moreover, without loss of generality we can assume that $\mu^-_n$ and $\mu^+_n$ are also probability measures and they are mutually singular. Indeed, assuming the validity of Theorem \ref{thm:main} in this special case, it is easy to deduce its validity in general, using the following argument. Denoting $\nu^-_n$ and $\nu^+_n$ respectively the negative and the positive part of the measure $\mu^+_n-\mu^-_n$, since the supports of $\mu^-$ and $\mu^+$ are disjoint, we have
that $\nu^-_n\rightharpoonup \mu^-$ and $\nu^+_n\rightharpoonup \mu^+$. Moreover, since the ambient is a compact set, $\nu^\pm_n(X)\to\mu^\pm(X)=1$. Now, denoting $\eta_n:=\nu^-_n(X)=\nu^+_n(X)$, we are in the poisiton to apply Theorem \ref{thm:main} in the special case above for the approximating measures $\eta_n^{-1}\nu^\pm_n$, the limiting measures $\mu^\pm$, the optimal traffic paths $\eta_n^{-1}T_n$ and the limit traffic path $T$. Since $\eta_n\to 1$ and $T_n\rightharpoonup T$ is in force, then $\eta_n^{-1}T_n\rightharpoonup T$ is satisfied.

~
\newline
By contradiction, we assume $T$ is not optimal, i.e. 
\begin{equation}\label{e:dropmassa}
\MM(T)\geq \MM(T_{opt})+\Delta,
\end{equation}
for some $\Delta>0$ and for some $T_{opt}$ with $\partial_\pm T_{opt}=\partial_\pm T$.

\smallskip
{\it Step 1: construction of the coverings of $A^-$ and $A^+$}. Let $C_{\alpha,d}$ be the constant in Lemma~\ref{irrigationspher}.
We claim that there exists a (finite or countable) family of balls $\{B_i^\pm = B(x^\pm_i,r^\pm_i)\}_{i \in I^\pm}$ covering respectively $A^-\cap \supp(\mu^-)$ and $A^+\cap \supp(\mu^+)$, such that
\begin{equation}\label{eqn:cov-limit-1}
\overline{\Big( \bigcup_{i\in I^-} B^-_i \Big)} \cap \overline{\Big(\bigcup_{i\in I^+} B^+_i\Big)} = \emptyset,
\end{equation}
\begin{equation}\label{eqn:cov-limit-2}
\sum_{i\in I^\pm} r^\pm_i < \frac{\Delta}{128 C_{\alpha,d}},	
\end{equation}\begin{equation}\label{eqn:cov-limit-3}
\MM\Big(T\trace \bigcup_{i\in I^\pm} \overline{B_i^\pm}  \Big) \leq \frac{\Delta}{128}, \qquad \MM\Big(T_{opt} \trace \bigcup_{i\in I^\pm} \overline{B_i^\pm}  \Big) \leq \frac{\Delta}{128},
\end{equation}\begin{equation}\label{eqn:cov-limit-4}
\mu^\pm(\partial B^\pm_i) =0, \qquad \mu^\pm_n(\partial B^\pm_i) =0 \qquad \forall i\in I^\pm, \, n \in \N.
\end{equation}
For simplicity, we assume $I^\pm$ to be either $\N$ or a set of the form $\{1,..., M^\pm\}$.
Finally, up to removing certain balls, we can assume the two coverings to be not redundant, namely, we can assume that
\begin{equation}
\label{eqn:discard}
\mu^\pm \Big(B_{i}^\pm \setminus \bigcup_{1 \leq j < i} {B_j^\pm} \Big) \neq 0, \qquad \forall i\in I^\pm.
\end{equation}
Since we have removed only balls that do not carry measure, the new set of balls still covers $A^-\cap \supp(\mu^-)$ and $A^+\cap \supp(\mu^+)$ up to a set of $\mu^\pm$-measure $0$. 

We now prove the claim of this Step 1. Let $d_0$ be the distance between $\supp (\mu^-)$ and $\supp (\mu^+)$, which is positive since the supports $ \supp( \mu^-) $ and $ \supp( \mu^+) $ are compact and disjoint. Applying Lemma~\ref{lemma:covering-other} with $\e =\min\{ \Delta/(128 C_{\alpha,d}), \Delta/128, d_0/4\}$ and $T'=T_{opt}$, we can find two finite coverings satisfying 
 \eqref{eqn:cov-limit-1}, \eqref{eqn:cov-limit-2}, \eqref{eqn:cov-limit-3}, and \eqref{eqn:cov-limit-4}.

\smallskip
{\it Step 2: choice of $N^\pm$}. Let $\e_1 >0$ to be chosen later. We choose $N^\pm$ satisfying
$$\mu^\pm \Big( \bigcup_{j=1}^{N^\pm} {B_j^\pm}\Big) > 1-\frac {\e_1} 4.
$$
\smallskip
{\it Step 3: choice of $n$}. Let $\e_2>0$ to be chosen later. For every $i\in  I^\pm$ we define 
$$ C_{i}^\pm=B_i^\pm \setminus \Big(\cup_{j=1}^{i-1}B_j^\pm \Big).$$ 
By \eqref{eqn:discard} the coverings are not redundant, that is, for every $i\in I^\pm$, 
\begin{equation}\label{nonzero}
\mu^\pm(C_i^\pm)>0.
\end{equation}
We claim that we can fix $n$ large enough so that the following properties hold:
\begin{equation}\label{eqn:choice-n-2}
\Flat (T_n - T) \leq \e_2,
\end{equation}
\begin{equation}\label{eqn:choice-n-3}
 \mu^\pm_n(C_i^\pm)\leq (1+\e_2)\mu^\pm(C_i^\pm), \qquad \forall i=1,\cdots, N^\pm,
\end{equation}
\begin{equation}\label{eqn:choice-n-4}
\mu^\pm_n\Big(\R^d \setminus \bigcup_{i=1}^{N^\pm} C_i^\pm \Big)\leq \frac {\e_1}2.
\end{equation}
Indeed, since $T_n\in \OTP(\mu^-_n,\mu^+_n)$, by Theorem \ref{ottimo_acicl}, Theorem \ref{s-decompcurr}  and Proposition~\ref{p:propr_good_dec}(2), $T_n=T_n[E_n,\tau_n,\theta_n]$ admits a good decomposition $\pi_n\in \mathscr{P}(\Lip)$ and its multiplicity $\theta_n$ verifies $\theta_n \leq 1$.
Consequently we get
$$\Mass(T_n)= \int_{E_n} \theta_n(x) d\Haus^1(x)\leq \int_{E_n} \theta_n^\alpha(x) d\Haus^1(x) \leq \MM(T_n)\leq C.$$
Moreover
$$\Mass(\partial T_n)=\mu^-_n(\R^d)+\mu^+_n(\R^d)=2<+\infty.$$
By the discussion after the definition of flat norm \eqref{e:flat}, the uniform bounds on the mass of the currents $T_n$ and on the mass of their boundaries guarantees that the weak$^*$ convergence implies \eqref{eqn:choice-n-2}, for $n$ sufficiently large. By \eqref{eqn:cov-limit-4} and since $\mu^\pm_n=\partial_\pm T_n$ weakly converges to $\mu^\pm= \partial_\pm T$, we observe that
$$\mu^\pm( \partial C_i) = 0 , \qquad \forall i=1,\cdots, N^\pm$$
and therefore
$$\lim_{n\to \infty}  \mu^\pm_n(C_i^\pm) =  \mu^\pm(C_i^\pm), \qquad \forall i=1,\cdots, N^\pm.$$
Since the right-hand side in the previous equality is non-zero thanks to \eqref{nonzero}, we obtain \eqref{eqn:choice-n-3} for $n$ large enough.

We fix $n$ large enough to satisfy the conditions in this step. Up to the end of the proof, we will always refer to this choice of $n$.
\smallskip

{\it Step 4: good decomposition of $T_n$ and selection.}
Let us define
\begin{equation}
\label{def:sel}
\pi_n^{sel}:= \pi_n \trace\Big\{\gamma: \gamma(0) \in \bigcup_{i=1}^{N^-} C_i^- \mbox{ and } \gamma(\infty) \in \bigcup_{i=1}^{N^+} C_i^+
 \Big\}.
 \end{equation}
Let us consider $T_n^{sel}$ to be the 1-dimensional current obtained from $T_n$ selecting only those curves that begin  inside the first $N^-$ balls and end inside the first $N^+$ balls, i.e.
$$T_n^{sel} := \int _{\Lip} R_\gamma \, d\pi_n^{sel} (\gamma).$$
Notice that, by Proposition~\ref{p:propr_good_dec}(3), $\pi_n^{sel}$ is a good decomposition of $T_n^{sel}$; in particular by Proposition~\ref{p:propr_good_dec}(1)
$$\partial_- T^{sel}_n =  \int _{\Lip}\delta_{\gamma(0)} \, d\pi_n^{sel} (\gamma)$$
is supported on $\bigcup_{i=1}^{N^-} C_i^-$ and it satisfies $\partial_- T^{sel}_n\leq\partial_- T_n=\mu_n^-$. 
 
For the same reason, $\pi_n - \pi_n^{sel}$ is a good decomposition of $T_n-T_n^{sel}$ and, denoting by $\tilde \theta_n$ the multiplicity of $T_n-T_n^{sel}$, we have the bound  
\begin{equation}
\label{eqn:mult}
\tilde \theta_n \leq \min\{\theta_n,(\pi_n-\pi_n^{sel})(\Lip) \}.
\end{equation}
Next we estimate
\begin{equation}
\label{eqn:natale}
\begin{split}
(\pi_n-\pi_n^{sel}&)(\Lip)= \pi_n \Big(\Big\{\gamma: \gamma(0) \not\in \bigcup_{i=1}^{N^-} C_i^- \mbox{ or } \gamma(\infty) \not\in \bigcup_{i=1}^{N^+} C_i^+
\Big\}\Big)
\\&
\leq\pi_n \Big(\Big\{\gamma: \gamma(0) \not\in \bigcup_{i=1}^{N^-} C_i^- \Big\}\Big)+ \pi_n \Big(\Big\{\gamma:  \gamma(\infty) \not\in \bigcup_{i=1}^{N^+} C_i^+
\Big\}\Big).
\end{split}
\end{equation}
By the good decomposition of $T_n$ (and in particular by \eqref{buona-dec-boundary}) for every Borel set $A \subseteq \R^d$
$$\pi_n\big(\{\gamma: \gamma(0)\in A\}\big) = \partial_-T_n(A)=\mu_n^-(A);
$$
hence, by \eqref{eqn:choice-n-4}
 $$
 \pi_n \Big(\Big\{\gamma: \gamma(0) \not\in \bigcup_{i=1}^{N^-} C_i^- \Big\} \Big) = \mu^-_n\Big( \Big(\bigcup_{i=1}^{N^-} C_i^-\Big)^c \Big)\leq \e_1/2.
 $$
A similar inequality holds for the second term in the right-hand side of \eqref{eqn:natale}. Overall,
it follows 
\begin{equation}\label{eqn:poca-massa-fuori-sel}
(\pi_n-\pi_n^{sel})(\Lip)\leq\varepsilon_1.
\end{equation}

%

We also notice that $T_n$ and $T_n^{sel}$ are close in flat norm by \eqref{eqn:mult} and \eqref{eqn:poca-massa-fuori-sel}
\begin{equation}
\label{eqn:flat-sel}
\begin{split}
\Flat(T_n-T_n^{sel}) &\leq \Mass (T_n-T_n^{sel}) = \int_{E_n} \tilde \theta_n d\Haus^1
\\&\leq \e_1^{1-\alpha}  \int_{E_n} \tilde \theta_n^\alpha d\Haus^1
\leq \e_1^{1-\alpha}  \int_{E_n}  \theta_n^\alpha d\Haus^1
\leq C\e_1^{1-\alpha}.
\end{split}
\end{equation}

\smallskip
{\it Step 5: restriction of $T_n$ inside the covering.}
We decompose $\pi_n^{sel}$ into the sum of finitely many, pairwise singular measures $\pi_{n,i}^{sel,-}$, according to the starting points of the associated curves, i.e. for every $i= 1,..., N^-$ we denote
\begin{equation}
\label{pezzo1}
\pi_{n,i}^{sel,-}:= \pi^{sel}_n \trace\Big\{\gamma: \gamma(0) \in C_i^-\Big\},
\end{equation}
and we notice that, using \eqref{def:sel},
\begin{equation}
\label{eqn:pi-n-i}
\sum_{i=1}^{N^-} \pi_{n,i}^{sel,-} = \pi_n^{sel}.
\end{equation}
We ``cut'' the current $T_n^{sel}$ considering the curves in its decomposition only up to the first time when they leave the ball where they begin, i.e. we define
\begin{equation}
\label{pezzo12}
T_{n,i}^{sel,-}:= \int _{\Lip} R_{res(0, O_{B_i^-})(\gamma)} \, d\pi^{sel,-}_{n, i} (\gamma), \qquad T_n^{sel,-}:= \sum_{i=1}^{N^{-}} T_{n,i}^{sel,-}.
\end{equation}
The measure
$$\sum_{i=1}^{N^{-}}(res(0, O_{B_i^-})(\cdot))_\sharp \pi^{sel,-}_{n, i}$$
is a good decomposition of $T_n^{sel,-}$: this is a consequence of Remark~\ref{remark:A-B-disj} applied to $I(\gamma):=\gamma(0)$, 
$$F(\gamma):=\begin{cases} O_{B_i^-}(\gamma), & \mbox{if }  \gamma(0) \in C_i^-, \mbox { for some $i=1,\cdots, N^-$}\\ 0, & \mbox{otherwise}\end{cases},$$ 
$E^-:= (\cup_{i=1}^{N^-}B^-_i)\setminus (\cup_{i=1}^{N^-}\partial B^-_i)$ and $E^+:=\cup_{i=1}^{N^-}\partial B^-_i $. Notice that the assumption of the Remark are satisfied in view of \eqref{eqn:cov-limit-4}.

Using this fact, by \eqref{buona-dec-boundary}, \eqref{eqn:pi-n-i} and \eqref{pezzo12}, we get
\begin{equation}
\label{obs1}
\partial_-T_n^{sel,-}= \partial_-T_n^{sel}.
\end{equation}
 Analogously we define 
\begin{equation}
\label{pezzo2}
\pi_{n,j}^{sel,+}:= \pi^{sel}_n \trace\Big\{\gamma:\gamma(\infty) \in C_j^+\Big\} \qquad\mbox{for every $j= 1,..., N^+$},
\end{equation}
and we ``cut'' the current $T_n^{sel}$ considering the curves in its decomposition only from the last time when they enter in the ball where they end, i.e. we define
\begin{equation}
\label{pezzo13}
T_{n,j}^{sel,+}:= \int _{\Lip} R_{res(E_{B_j^+}, \infty)(\gamma)} \, d\pi^{sel,+}_{n, j} (\gamma), \qquad T_n^{sel,+}= \sum_{j=1}^{N^{+}} T_{n,j}^{sel,+}. 
\end{equation}
Arguing as for \eqref{obs1}, we get
\begin{equation}
\label{obs2}
\partial_+T_n^{sel,+}= \partial_+T_n^{sel},
\end{equation}
and combining \eqref{obs1} and \eqref{obs2}, we derive
\begin{equation}\label{bordi}
\partial T_n^{sel,-}+\partial T_n^{sel,+}=\partial T_n^{sel}+ \partial_+ T_{n}^{sel,-}- \partial_- T_{n}^{sel,+}.
\end{equation}

\smallskip
{\it Step 6: good decomposition of $T_{opt}$ and restriction outside the covering.}
Let $\pi_{opt}$ be a good decomposition of $T_{opt}$.
Let us decompose $\pi_{opt}$ into the sum of countably many, mutually singular measures $\pi_{opt, i, j }$, according to the starting and the ending points of the associated curves, i.e., for every $i \in I^-$ and $j\in I^+$ we denote
$$\pi_{opt, i, j }:= \pi \trace \Big\{\gamma:  \gamma(0) \in C_i^- \mbox{ and } \gamma(\infty) \in C_j^+
\Big\}. $$
We denote by $T_{opt,i,j}$ the traffic path associated to $\pi_{opt,i,j}$. Now we ``cut'' the current $T_{opt}$ considering the curves in its decomposition only from the first time when they leave the ball where they begin, up to the last time when they enter in the ball where they end, i.e. we define (see \S \ref{s:restriction})
$$T_{opt,i,j}^{restr}:=\int _{\Lip} R_{res(O_{B_i^-}, E_{B_j^+})(\gamma)} \, d\pi_{opt, i, j} (\gamma),\qquad T_{opt}^{restr} := \sum_{i \in I^-, j\in I^+}T_{opt,i,j}^{restr}.$$
Notice that, by Remark~\ref{remark:A-B-disj} and \eqref{eqn:cov-limit-1}, this formula gives a good decomposition of $T_{opt}^{restr}$. Here we use the same abuse of notation, as in \eqref{gooddec1}.
By Proposition~\ref{p:propr_good_dec_2}, we have that the multiplicity of $T_{opt}^{restr}$ is pointwise bounded by the multiplicity of $T_{opt}$, so that
\begin{equation}
\label{eqn:en-T-opt-restr-fuori}
\MM\Big( T_{opt}^{restr} 
\trace  \Big( \big( \cup_{i=1}^{N^-} \overline B^-_i \big) \cup \big(\cup_{i=1}^{N^+} \overline  B^+_i\big)\Big)^c \Big) \leq \MM (T_{opt}),
\end{equation}
and by \eqref{eqn:cov-limit-3}
\begin{equation}
\label{eqn:en-T-opt-restr-dentro}
\MM\Big( T_{opt}^{restr}  \trace  \big( \cup_{i=1}^{N^\pm} \overline B^\pm_i \big) \Big)
\leq \MM\Big( T_{opt}  \trace  \big( \cup_{i=1}^{N^\pm} \overline B^\pm_i \big)\Big)
 \leq \frac{\Delta}{128}.
\end{equation}
We observe that:
\begin{equation}
\label{marginal}
\sum_{j\in I^+}\partial_-T_{opt,i,j}^{restr}(\partial B_i^-)=\sum_{j\in I^+}\partial_-T_{opt,i,j}(C_i^-)=\partial_-T_{opt}(C_i^-)=\mu^-(C_i^-),
\end{equation}
where the first equality follows because the first (resp. second) term can be seen as the total mass of the positive (resp. negative) part of the boundary of
$$\sum_{j\in I^+}\int _{\Lip} R_{(0, res(O_{B_i^-}))(\gamma)} \, d\pi_{opt, i, j} (\gamma).$$ 
This is true because, by Remark~\ref{remark:A-B-disj} and \eqref{eqn:cov-limit-4}, this formula gives a good decomposition (with the usual abuse of notation).

\smallskip
{\it Step 7: connection along the spheres.}
By Proposition~\ref{p:propr_good_dec}(1) we have $\partial_\pm T_n^{sel}\leq\partial_\pm T_n=\mu_n^\pm$. We deduce that
\begin{equation}\label{e:bordomeno}
\begin{split}
\mu_n^-(C_i^-)&\geq \partial_- T^{sel}_{n}(C_i^-)\overset{\eqref{obs1}}{=}\partial_- T^{sel,-}_{n}(C_i^-)\\
&\overset{\eqref{pezzo1}}{=} \partial_- T^{sel,-}_{n,i} (\R^d)\overset{\eqref{e:massbord}}{=}\partial_+ T^{sel,-}_{n,i} (\R^d) = \partial_+ T^{sel,-}_{n,i} (\partial B^-_i)
\end{split}
\end{equation}
and similarly
\begin{equation}\label{e:bordopiu}
\mu_n^+(C_j^+)\geq \partial_- T^{sel,+}_{n,j} (\partial B^+_j).
\end{equation}
Combining this with \eqref{eqn:choice-n-3}, it follows that, for every $i\in I^-$,
$$\partial_+ T^{sel,-}_{n,i} (\partial B^-_i)\leq(1+\e_2)\mu^-(C^-_i)\overset{\eqref{marginal}}{=}(1+\e_2) \sum_{j\in I^+}\partial_-T_{opt,i,j}^{restr}(\partial B^-_i)$$
and analogously, for every $j\in I^+$,
$$\partial_- T^{sel,+}_{n,j} (\partial B^+_j)\leq(1+\e_2)\mu^+(C^+_j)=(1+\e_2) \sum_{i\in I^-}\partial_+T_{opt,i,j}^{restr}(\partial B^+_j).$$
Hence, for every $i\in I^-$, we denote
\begin{equation}\label{alphai}
\alpha_i^-:=\frac{\partial_+ T^{sel,-}_{n,i} (\partial B^-_i)}{(1+\e_2) \sum_{j\in I^+}\partial_-T_{opt,i,j}^{restr}(\partial B^-_i)} \in [0,1]
\end{equation}
and, for every $j\in I^+$,
\begin{equation}\label{alphaj}
\alpha_j^+:=\frac{\partial_- T^{sel,+}_{n,j} (\partial B^+_j)}{(1+\e_2) \sum_{i\in I^-}\partial_+T_{opt,i,j}^{restr}(\partial B^+_j)} \in [0,1].
\end{equation}
We define 
\begin{equation}
\label{marginale1}
T^{conn, -}_{n,i} \in \TP\Big(\partial_+ T^{sel,-}_{n,i},\alpha^-_i(1+\e_2) \sum_{j\in I^+}\partial_-T_{opt,i,j}^{restr}\Big)
\end{equation}
to be the traffic path given by Lemma~\ref{irrigationspher} (supported on $\partial B_i^-$). The lemma can be applied since the two marginals in \eqref{marginale1} are supported on $\partial B_i^-$ and they have same total mass, as a consequence of \eqref{alphai}. Its cost is estimated by
\begin{equation}
\label{eqn:t-conn-i-cost}
\MM \big(T^{conn, -}_{n,i} \big) \leq 
C_{\alpha,d} \big(\partial_+ T^{sel,-}_{n,i} (\partial B^-_i) \big)^\alpha r_i^-
\leq 
C_{\alpha,d} r_i^-.
\end{equation}
Analogously, we define a traffic path
\begin{equation}
\label{marginale2}
T^{conn, +}_{n,j}\in \TP\Big(\alpha^+_j(1+\e_2) \sum_{i\in I^-}\partial_+T_{opt,i,j}^{restr}, \partial_- T^{sel,+}_{n,j}\Big),
\end{equation}
supported on $\partial B_j^+$, whose cost is again estimated by
\begin{equation}
\label{eqn:t-conn-i-cost2}
\MM \big(T^{conn, +}_{n,j} \big) \leq 
C_{\alpha,d} \big(\partial_- T^{sel,+}_{n,j} (\partial B^+_j) \big)^\alpha r_j^+
\leq 
C_{\alpha,d} r_j^+.
\end{equation}
Finally, we define the traffic paths
$$ T_n^{conn,-}:= \sum_{i=1}^{N^{-}} T_{n,i}^{conn,-} \qquad \mbox{and}\qquad T_n^{conn,+}:= \sum_{j=1}^{N^{+}} T_{n,j}^{conn,+}.$$
We denote
\begin{equation}
\label{marginale33}
\sigma^+_n:= (1+\e_2)\sum_{i=1}^{N^{-}} \alpha^-_i \sum_{j\in I^+}\partial_-T_{opt,i,j}^{restr},
\end{equation}
\begin{equation}
\label{marginale44}
\sigma^-_n:= (1+\e_2) \sum_{j=1}^{N^{+}} \alpha^+_j \sum_{i\in I^-}\partial_+T_{opt,i,j}^{restr},
\end{equation}
from \eqref{marginale1}, \eqref{marginale2}, \eqref{pezzo12} and \eqref{pezzo13}, we infer
\begin{equation}
\label{marginale3}
T^{conn, -}_{n} \in \TP(\partial_+ T^{sel,-}_{n},\sigma^+_n),\qquad \mbox{and} \qquad T^{conn, +}_{n}\in \TP( \sigma^-_n, \partial_- T^{sel,+}_{n}).
\end{equation}
Using the fact that $\pi_n \in \mathscr{P}(\Lip)$, one gets
\begin{equation}
\label{remarkmarg}
\begin{split}
\sigma^+_n(\R^d)&\overset{\eqref{marginale3},\eqref{e:massborduno}}{=} \partial_+ T^{sel,-}_{n}(\R^d)\overset{\eqref{e:massbord}}{=}\partial_- T^{sel,-}_{n}(\R^d)\\
&\overset{\eqref{obs1}}{=}\partial_- T^{sel}_{n}(\R^d)\overset{\eqref{buona-dec-boundary}}{=}\pi^{sel}_{n}(\Lip)\overset{\eqref{eqn:poca-massa-fuori-sel}}{\geq}1-\e_1.
\end{split}
\end{equation}
Using the sub-additivity of the $\alpha$-mass, we get the energy estimate
\begin{equation}
\label{eqn:energy-comp}
\MM \big(T^{conn, \pm}_{n} \big) \overset{\eqref{eqn:t-conn-i-cost}, \eqref{eqn:t-conn-i-cost2}}{\leq}  \sum_{i=1}^{N^{\pm}}
C_{\alpha,d} r_i^\pm \overset{\eqref{eqn:cov-limit-2}}{\leq} \frac{\Delta}{128}.
\end{equation}

\smallskip
{\it Step 8: bringing back the mass in excess.} Denoting 
\begin{equation}
\label{bordi3}\nu^-:=(1+\e_2) \sum_{i\in I^-, j\in I^+}\partial_-T_{opt,i,j}^{restr},\quad \mbox{and} \quad \nu^+:=(1+\e_2) \sum_{i\in I^-, j\in I^+}\partial_+T_{opt,i,j}^{restr},
\end{equation}
we get that
\begin{equation}
\label{bordi10}
(1+\e_2)T_{opt}^{restr} \in \TP(\nu^-,\nu^+).
\end{equation}
We define the two non-negative measures
\begin{equation}
\label{eqn:pezzi}
\nu^-_n:=\nu^--\sigma^+_n, \qquad \nu^+_n:=\nu^+- \sigma^-_n.
\end{equation}
Since by \eqref{alphai}, \eqref{alphaj} $\alpha_i^-, \alpha_j^+ \in [0,1]$, comparing \eqref{marginale3} with  \eqref{bordi3}, we get
\begin{equation}
\label{sottomisure}
\sigma^+_n\leq \nu^-,\quad \sigma^-_n\leq \nu^+, \quad \nu_n^-\leq\nu^-,\quad\mbox{ and }\quad \nu_n^+\leq\nu^+.
\end{equation}
We claim that
\begin{equation}
\label{claim}
\nu_n^-(\R^d)=\nu_n^+(\R^d)\leq \e_1+\e_2.
\end{equation}
Indeed we can compute
\begin{equation*}
\begin{split}
\sigma^+_n(\R^d)&\overset{\eqref{marginale3},\eqref{e:massborduno}}{=}\partial_+ T^{sel,-}_{n}(\R^d)\overset{\eqref{e:massbord}}{=}\partial_- T^{sel,-}_n(\R^d)\overset{\eqref{obs1}}{=}\partial_- T^{sel}_n(\R^d)\\
&\overset{\eqref{e:massbord}}{=}\partial_+ T^{sel}_n(\R^d)\overset{\eqref{obs2}}{=}\partial_+ T^{sel,+}_n(\R^d)\overset{\eqref{e:massbord}}{=}\partial_- T^{sel,+}_{n}(\R^d)\overset{\eqref{marginale3},\eqref{e:massborduno}}{=}\sigma^-_n(\R^d),
\end{split}
\end{equation*}
which, together with \eqref{eqn:pezzi} and the fact that $\nu^-(\R^d)=\nu^+(\R^d)$, implies $  \nu_n^-(\R^d)=\nu_n^+(\R^d)$.
Since $\sigma^+_n\leq \nu^-$, we can estimate
\begin{equation*}
\begin{split}
\nu_n^-(\R^d)&=\nu^-(\R^d)-\sigma^+_n(\R^d)\overset{\eqref{bordi3},\eqref{marginal}}{\leq}(1+\e_2)\mu^-(\R^d)-\sigma^+_n(\R^d)\\
&\overset{\eqref{remarkmarg}}{\leq}(1+\e_2)-(1-\e_1)=\e_1+\e_2,
\end{split}
\end{equation*}
getting the claim \eqref{claim}.

Therefore, by \eqref{sottomisure}, \eqref{claim} and \eqref{bordi10}, we can apply Proposition~\ref{l:small-transport} to prove the existence of a path 
\begin{equation}
\label{bordi2}
T^{back}\in\TP(\nu_n^+,\nu_n^-)
\end{equation}
 with 
\begin{equation}
\label{eqn:energy-t-back}
\Mass^\alpha(T^{back})\leq \frac{\Delta}{128},
\end{equation}
provided $\e_1$ and $\e_2$ are chosen small enough.

From \eqref{marginale3}, \eqref{bordi10}, \eqref{bordi2}, and \eqref{eqn:pezzi} we compute
\begin{equation}\label{bordi6}
\begin{split}
&\partial T_n^{conn,-}+(1+\varepsilon_2)\partial T_{opt}^{restr}+\partial T^{back}+\partial T_n^{conn,+}\\
&=\sigma^+_n-\partial_+ T_n^{sel,-}+\nu^+-\nu^-+\nu_n^--\nu_n^++\partial_- T_n^{sel,+}-\sigma^-_n\\
&=\partial_- T_n^{sel,+}-\partial_+ T_n^{sel,-}.
\end{split}
\end{equation}

\smallskip

{\it Step 9: definition of a competitor for $T_n^{sel}$.}
Eventually we define 
$$\tilde{T}_n^{sel}:=T_n^{sel,-}+T_n^{conn,-}+(1+\varepsilon_2)T_{opt}^{restr}+T^{back}+T_n^{conn,+}+T_n^{sel,+}.$$
We show that it has the same boundary of $T_n^{sel}$
\begin{equation}
\label{eqn:boundary-tilde-t-n-sel}
\partial \tilde{T}_n^{sel} = \partial T_n^{sel}.
\end{equation}
Indeed, using \eqref{bordi} and \eqref{bordi6}, we get
\begin{equation}\label{competitore}
\begin{split}
\partial \tilde{T}_n^{sel}&=\partial T_n^{sel,-}+\partial T_n^{conn,-}+(1+\varepsilon_2)\partial T_{opt}^{restr}+\partial T^{back}+\partial T_n^{conn,+}+\partial T_n^{sel,+}\\
&\overset{\eqref{bordi}, \eqref{bordi6}}{=}\partial T_n^{sel}+\partial_+ T_{n}^{sel,-}- \partial_- T_{n}^{sel,+}+\partial_- T_n^{sel,+}-\partial_+ T_n^{sel,-}=\partial T_n^{sel}.
\end{split}
\end{equation}

\smallskip

{\it Step 10: estimates on the energy of the competitor.} 
In the following we denote by $U$ the union of our two closed coverings
$$U^\pm  := \cup_{i=1}^{N^\pm} \overline B^\pm_i\qquad U := U^+ \cup U^-.$$
We claim that the competitor $\tilde T_n^{sel}$ for $T_n^{sel}$ enjoys the following estimate
\begin{equation}
\label{eqn:maestoso-competitore1}
\MM\Big( \tilde T_n^{sel}\trace  U^c \Big) \leq \MM (T_{opt}) +\frac \Delta {4}
\end{equation}
and that
\begin{equation}
\label{eqn:maestoso-competitore2}
\MM\Big( \big(\tilde T_n^{sel} - T_n^{sel, \pm}\big) \trace  U^\pm\Big) \leq \frac{\Delta}{32}.
\end{equation}

We first focus on \eqref{eqn:maestoso-competitore1}. By their definition, the currents $T_n^{conn,\pm}$, $T_n^{sel,\pm}$ are supported on the sets $ U^\pm$; hence, they are supported on sets disjoint from $U^c$. Using \eqref{eqn:en-T-opt-restr-fuori} and $C\e_2 \leq \frac{\Delta}{8}$, we can compute
\begin{equation}
\begin{split}
\MM\Big( \tilde T_n^{sel}\trace  U^c \Big) &= \MM( \big((1+\varepsilon_2)T_{opt}^{restr}+T^{back} \big)\trace  U^c )\\
&
\leq (1+\varepsilon_2)^\alpha \MM (T_{opt}^{restr}\trace  U^c )+ \MM(T^{back}) \\
&\leq  \MM (T_{opt}) +C\e_2 +\frac \Delta {128}\leq  \MM (T_{opt})+\frac \Delta {4}.
\end{split}
\end{equation}
To prove \eqref{eqn:maestoso-competitore2} (we show it for the choice $\pm = -$), it is enough to show that
\begin{equation}
\label{eqn:maestoso-competitore1-}
\MM\big( \big(\tilde T_n^{sel} - T_n^{sel, -}\big)\trace U^- \big) \leq \frac \Delta {32}.
\end{equation}
Using again that the currents $T_n^{conn,+}$, $T_n^{sel,+}$ are supported on the set $ U^+$, we estimate, by the subadditivity of the $\alpha$-mass,
\begin{equation*}
\begin{split}
\MM\Big( \big(&\tilde T_n^{sel} - T_n^{sel, -}\big)\trace U^- \Big) \leq \MM \big(T_n^{conn,-}\trace U^-+(1+\varepsilon_2)T_{opt}^{restr}\trace U^-+T^{back}\trace U^- \big)\\
&\leq \MM \big(T_n^{conn,-}\big)+(1+\varepsilon_2)^\alpha \MM \big(T_{opt}^{restr}\trace U^-)+\MM \big(T^{back} \big)\leq \frac \Delta {32},
\end{split}
\end{equation*}
where in the last inequality we used $\e_2 \leq 1/4$, \eqref{eqn:en-T-opt-restr-dentro}, \eqref{eqn:energy-comp}, \eqref{eqn:energy-t-back}. This concludes the proof of \eqref{eqn:maestoso-competitore2}.

\smallskip

{\it Step 11: definition of a competitor for $T_n$.}
We define $\overline{T}_n:= \tilde T_n^{sel} + T_n -  T_n^{sel} $ as a competitor for the $\alpha$-mass optimizer $T_n$, with the aim to prove that the former has less $\alpha$-mass than the latter.
Indeed, by \eqref{competitore}, $\partial  \overline{T}_n =\partial  T_n$ and consequently
$$\tilde T_n^{sel} + T_n -  T_n^{sel} \in \TP(\mu^-_n,\mu^+_n).$$
 We split its energy as
\begin{equation}
\label{eqn:est-in-and-out}
\MM ( \overline{T}_n ) =
\MM \big(\overline{T}_n \trace U \big)
+
\MM \big(\overline{T}_n \trace U^c \big)
\end{equation}
For the first term, the additivity of the $\alpha$-mass on disjoint sets gives
\begin{equation}
\label{eqn:correction}
\MM \big(\overline{T}_n \trace U \big)
= \MM \big(\big( \tilde T_n^{sel} + T_n -  T_n^{sel}\big) \trace U^+ \big) + \MM \big(\big( \tilde T_n^{sel} + T_n -  T_n^{sel}\big) \trace U^- \big).
\end{equation}
We estimate each term by means of  \eqref{eqn:maestoso-competitore2}; since the proof is the same, we do it for the first term in the right-hand side
\begin{equation}
\label{eqn:est-in1}
\begin{split}
&\MM \big(\big( \tilde T_n^{sel} + T_n -  T_n^{sel}\big) \trace U^- \big)
\\ &\leq
 \MM\big( \big(\tilde T_n^{sel} - T_n^{sel,-}
 \big) \trace  U^- \big) + \MM\big( \big(T_n^{sel,-}+T_n-  T_n^{sel} \big) \trace U^-\big) 
 \\&
 \leq \frac{\Delta}{32}+ \MM\big( \big( T_n^{sel,-}+T_n-  T_n^{sel} \big) \trace U^-\big)
\end{split}
\end{equation}

The latter can be estimated by noticing that it is a ``part of an optimum'' with 
\begin{equation}
\label{eqn:inside1}
\MM\big( \big(T_n^{sel,-}+T_n-  T_n^{sel}\big) \trace U^-\big) \leq 
\MM\big(T_n \trace U^-\big).
\end{equation}
Indeed we apply Proposition~\ref{p:propr_good_dec_2} with $T= T_n^{sel}$ and $\tilde T=T_n^{sel,-}$, to obtain that
$$T_n^{sel}-T_n^{sel,-}=\beta T_n^{sel}, \qquad \mbox{where } \beta:\R^d \to [0,1],$$
and Proposition~\ref{p:propr_good_dec}(3) with $T=T_n$ and $T'=T_n^{sel}$, to obtain that $T_n^{sel} = \varphi T_n$, where $\varphi:\R^d \to [0,1]$, and therefore
$$T_n-(T_n^{sel}-T_n^{sel,-})=T_n-\beta T_n^{sel}=(1-\varphi\beta) T_n, \qquad \mbox{where } [1-\varphi \beta]:\R^d \to [0,1].$$
We can conclude that
$$\MM\big( \big(T_n^{sel,-}+T_n-  T_n^{sel}\big) \trace U^-\big) \leq 
\sup_{x \in \R^d}\{1-\beta(x)\varphi(x)\}^\alpha\MM\big(T_n \trace U^-\big)\leq \MM\big(T_n \trace U^-\big),$$
which is exactly \eqref{eqn:inside1}.

Putting together \eqref{eqn:correction}, \eqref{eqn:est-in1}, \eqref{eqn:inside1}, we get an estimate for the first term in the right-hand side of \eqref{eqn:est-in-and-out}
\begin{equation}
\label{eqn:first-term-est}
\MM \big(\overline{T}_n \trace U \big)
\leq \frac{\Delta}{16}+ \MM\big(T_n \trace U\big).
\end{equation}

The second term in \eqref{eqn:est-in-and-out} can be instead estimated through the sub-additivity of the $\alpha$-mass, the energy bound on the competitor $\tilde T_n^{sel}$ in \eqref{eqn:maestoso-competitore1}, and the energy gap in \eqref{e:dropmassa}
\begin{equation}
\label{eqn:outside1}
\begin{split}
\MM \big(\overline{T}_n \trace U^c \big)& \leq 
\MM \big( \tilde T_n^{sel} \trace U^c \big)
+
\MM \big(\big( T_n -  T_n^{sel}\big) \trace U^c \big)
\\&
\leq \MM (T_{opt}) +\frac \Delta {4}
+
\MM \big(\big( T_n -  T_n^{sel}\big) \trace U^c \big)
\\&
\leq \MM (T) -\frac {3\Delta} {4}
+
\MM \big(\big( T_n -  T_n^{sel}\big) \trace U^c \big).
\end{split}
\end{equation}
We fix $\delta$ obtained from Lemma~\ref{high_multiplicity} with the choices $A=U^c$ and $\e = \Delta/8$. 
The conclusion of the lemma holds for $T'=T_n^{sel}$, provided we add the following further constraints on $\e_1$ and $\e_2$:
\begin{equation}\label{scelte_piccolezza}
\e_2 \leq \frac \delta 2, \qquad C \e_1^{1-\alpha} \leq \frac\delta 2, \qquad \e_1 \leq \frac \delta {4}, \qquad 16\e_1^\alpha C \leq \delta^\alpha\Delta.
\end{equation}
By sub-additivity of flat norm, \eqref{eqn:choice-n-2} and \eqref{eqn:flat-sel}, we find that 
$$\Flat(T_n^{sel}-T) \leq \Flat(T_n-T) + \Flat(T_n^{sel}-T_n) \leq \e_2+ C \e_1^{1-\alpha} \leq \delta.
$$
Using the previous inequality and Lemma~\ref{high_multiplicity}, 
\begin{equation*}
\begin{split}
\MM (T) &= \MM (T\trace U^c)+ \MM (T\trace U) \\
&\leq \MM \big(T_n^{sel} \trace \big(U^c \cap \{ \theta^{sel}_n> \delta \}\big)\big)+ \MM (T\trace U) +\frac {\Delta} {8}\\
&\overset{\eqref{eqn:cov-limit-3}}{\leq} \MM \big(T_n^{sel} \trace \big(U^c \cap \{ \theta^{sel}_n> \delta \}\big)\big)+ \frac {\Delta} {4}.
\end{split}
\end{equation*}
Substituting the previous inequality in \eqref{eqn:outside1}, we find
\begin{equation}
\label{eqn:outside2}
\begin{split}
\MM \big(\overline{T}_n  \trace U^c \big)\leq \MM \big(T_n^{sel} \trace \big(U^c \cap \{ \theta^{sel}_n> \delta \}\big)\big) -\frac {\Delta} {2}
+
\MM \big(\big( T_n -  T_n^{sel}\big) \trace U^c \big).
\end{split}
\end{equation}
We claim that it is possible to apply Lemma~\ref{lemma:quasiadditive} with $T_1=\big( T_n -  T_n^{sel}\big) \trace U^c $, $T_2=T_n^{sel} \trace \big(U^c \cap \{ \theta^{sel}_n> \delta \}\big)$, and $\e=\e_1/ \delta$. Indeed, by \eqref{eqn:mult} $T_n-T_n^{sel}$ has multiplicity less than or equal to $\e_1$ and by \eqref{scelte_piccolezza} we have $\delta \geq 4\e_1$.
Consequently, by \eqref{eqn:outside2}
\begin{equation}
\label{eqn:outside3}
\begin{split}
&\MM \big(\overline{T}_n  \trace U^c \big)\\
&\leq  \Big(1+4\Big( \frac{\e_1}{\delta} \Big)^{\alpha} \Big) \MM \big(T_n^{sel} \trace \big(U^c \cap \{ \theta^{sel}_n> \delta \}\big)+ \big( T_n -  T_n^{sel}\big) \trace U^c\big) -\frac {\Delta} {2}\\
&=\Big(1+4\Big( \frac{\e_1}{\delta} \Big)^{\alpha}\Big) \MM \big(\beta T_n \trace U^c\big) -\frac {\Delta} {2},
\end{split}
\end{equation}
where $\beta: \R^d \to [0,1]$.
Since by hypothesis $\MM \big(T_n  \big) \leq C$, using \eqref{scelte_piccolezza}, we find that 
\begin{equation}
\label{eqn:outside4}
\begin{split}
\MM \big(\overline{T}_n \trace U^c \big) \leq \MM \big(T_n \trace U^c \big) +4\Big( \frac{\e_1}{\delta}\Big)^\alpha C -\frac {\Delta} {2}\overset{\eqref{scelte_piccolezza}}{\leq} \MM \big(T_n \trace U^c \big) -\frac {\Delta} {4}.
\end{split}
\end{equation}

Putting together \eqref{eqn:first-term-est} and \eqref{eqn:outside4}, we find that
$$\MM ( \overline{T}_n )
\leq   \MM \big(T_n \trace U \big) +\MM \big(T_n \trace U^c \big) -\frac {\Delta} {8} < \MM (T_n),$$
which is a contradiction to the optimality of $T_n$.

\subsection*{Acknowledgements}
M. C. acknowledges the support of Dr. Max R\"ossler, of the Walter Haefner Foundation and of the ETH Z\"urich Foundation. A. D.R. is supported by SNF 159403 {\it Regularity questions in geometric measure theory}. A. M. is supported by the ERC-grant 306247 {\it Regularity of area minimizing
currents}.

%
%

%
%

\vskip .5 cm

{\parindent = 0 pt\begin{footnotesize}

M.C, A.D.R. \and A.M.
\\
Institut f\"ur Mathematik,
Mathematisch-naturwissenschaftliche Fakult\"at,
Universit\"at Z\"urich\\
Winterthurerstrasse 190,
CH-8057 Z\"urich,
Switzerland
\\
e-mail M.C.: {\tt maria.colombo@math.uzh.ch}\\
e-mail A.D.R.: {\tt antonio.derosa@math.uzh.ch}\\
e-mail A.M.: {\tt andrea.marchese@math.uzh.ch}

\end{footnotesize}
}

\end{document}